\documentclass[12pt]{article}

\input xy

\xyoption{all}

\usepackage[utf8]{inputenc}
\usepackage[english]{babel}
\usepackage[T1]{fontenc}

\usepackage{sectsty,tocloft,hyperref,stmaryrd}

\usepackage[svgnames,x11names]{xcolor}

 \hypersetup{
   colorlinks, linkcolor={Blue},
   citecolor={red}, urlcolor={DarkOliveGreen4}
  }

\usepackage{enumitem}
\usepackage{mdwlist}
\usepackage{amssymb,amsfonts,amsbsy,amsthm,amsmath,graphicx,epsfig,bbm}
\usepackage{
times,mathrsfs}
\usepackage{pgf,tikz}
\usetikzlibrary{arrows}
\usepackage{textcomp}
\usepackage[many]{tcolorbox} 

\partfont{\normalsize}

\subsectionfont{\normalsize}

\subsubsectionfont{\mdseries\itshape\normalsize}

\newenvironment{eq}{\begin{equation}}{\end{equation}} 

\theoremstyle{theorem}
\newtheorem{thm}{Theorem}[section]
\newtheorem{cor}{Corollary}[thm]

\newtheorem{lem}[thm]{Lemma}
\newtheorem{prop}[thm]{Proposition}

\newtheorem{quest}[thm]{Question}
\newtheorem{defi}[thm]{Definition}

\theoremstyle{definition}
\newtheorem{rmq}[thm]{Remark}
\newtheorem*{rmqs}{Remarks}
\newtheorem{exe}[thm]{Example}

\theoremstyle{remark}

\newenvironment{rmq*}
 {\pushQED{\qed}\rmq}
 {\popQED\endrmq}

\newenvironment{exe*}
 {\pushQED{\qed}\exe}
 {\popQED\endexe}

\renewcommand{\bf}[1]{\boldsymbol{#1}}
\renewcommand{\rm}[1]{\mathrm{#1}}

\newcommand{\bbE}{\mathbb{E}}

\newcommand{\bbN}{\mathbb{N}}
\newcommand{\bbP}{\mathbb{P}}
\newcommand{\bbQ}{\mathbb{Q}}
\newcommand{\bbR}{\mathbb{R}}
\newcommand{\bbT}{\mathbb{T}}
\newcommand{\bbZ}{\mathbb{Z}}

\newcommand{\bfX}{\mathbf{X}}
\newcommand{\bfY}{\mathbf{Y}}
\newcommand{\bfZ}{\mathbf{Z}}
\newcommand{\bfU}{\mathbf{U}}
\newcommand{\bfV}{\mathbf{V}}
\newcommand{\bfW}{\mathbf{W}}

\newcommand{\T}{\mathrm{T}}

\newcommand{\A}{\mathscr{A}}
\newcommand{\B}{\mathscr{B}}
\newcommand{\C}{\mathscr{C}}
\newcommand{\D}{\mathscr{D}}
\newcommand{\E}{\mathscr{E}}
\newcommand{\F}{\mathscr{F}}

\newcommand{\G}{\mathscr{G}}

\newcommand{\I}{\mathscr{I}}
\newcommand{\J}{\mathscr{J}}

\newcommand{\scrL}{\mathscr{L}}
\newcommand{\calL}{\mathcal{L}}

\newcommand{\M}{\mathscr{M}}
\renewcommand{\P}{\mathscr{P}}
\newcommand{\calP}{\mathcal{P}}

\newcommand{\scrT}{\mathscr{T}}

\renewcommand{\a}{\alpha}

\newcommand{\eps}{\varepsilon}
\newcommand{\g}{\gamma}

\renewcommand{\l}{\lambda}
\renewcommand{\o}{\omega}
\newcommand{\s}{\sigma}

\newcommand{\aut}{\mathrm{Aut}}

\renewcommand{\t}[1]{\widetilde{#1}}

\renewcommand{\to}{\longrightarrow}
\newcommand{\arr}{\rightarrow}
\renewcommand{\phi}{\varphi}

\newcommand*{\quotient}[2]
{\ensuremath{
    #1/\!\raisebox{-.65ex}{\ensuremath{#2}}}}

\newcommand{\indep}{\perp \!\!\! \perp}

\def\commutatif{\ar@{}[rd]|{\circlearrowleft}}

\definecolor{main}{HTML}{5989cf}
\newtcolorbox{boxB}{
    fontupper = \bf, 
    boxrule = 1.5pt,
    colframe = main,
    rounded corners,
    arc = 5pt   
}

\begin{document}

\title{Confined extensions and non-standard dynamical filtrations}

\author{Séverin Benzoni
\\}
\date{}

\maketitle

\begin{abstract}
In this paper, we explore various ways in which a factor $\s$-algebra $\B$ can sit in a dynamical system $\bfX :=(X, \A, \mu, T)$, i.e. we study some possible structures of the \emph{extension} $\A \arr \B$. We consider the concepts of \emph{super-innovations} and \emph{standardness} of extensions, which are inspired from the theory of filtrations. An important focus of our work is the introduction of the notion of \emph{confined extensions}, whose initial interest is that they have no super-innovation. We give several examples and study additional properties of confined extensions, including several lifting results. Then, using $T, T^{-1}$ transformations, we show our main result: the existence of non-standard extensions. Finally, this result finds an application to the study of dynamical filtrations, i.e. filtrations of the form $(\F_n)_{n \leq 0}$ such that each $\F_n$ is a factor $\s$-algebra. We show that there exist \emph{non-standard I-cosy dynamical filtrations}.
\end{abstract}

\tableofcontents 

\section{Introduction}

\subsection{Motivations}

The main focus of this paper is the study of the structure of a measure theoretic dynamical system with one of its factors. Given a dynamical system $\bfX := (X, \A, \mu, T)$, we consider its \emph{factor $\s$-algebras}, i.e. the sub-$\s$-algebras $\B \subset \A$ that are $T$-invariant. If $\B$ is a factor $\s$-algebra, we also say that $\A$ is an \emph{extension} of $\B$, or, in short, that the pair $(\A, \B)$ is an extension, which we note $\A \arr \B$. Our purpose is to understand some of the ways in which a factor $\s$-algebra can sit in $\A$. This is a key question in ergodic theory and has been studied from various points of view like, for example, in \cite{Furstenberg_disjointness}, \cite{complementability}, \cite{factors_ber_shifts}, \cite{stable_extensions} or \cite{thouvenot_ornstein_rel}. Here, our approach is largely inspired by the study of filtrations. In general, a filtration is an ordered family of $\s$-algebras, so we can view an extension as a filtration with only two steps. We use some vocabulary and notions from the theory of filtrations initiated by Vershik (see \cite{Vershik}, \cite{Laurent_standardness-cosiness}) and its adaptation to dynamical filtrations, i.e. filtrations made of factor $\s$-algebras on a dynamical system (see \cite{article_PL}, \cite{these_PL}). In return, our study of extensions enables us to get new results on dynamical filtrations. 

First consider what is left when we remove the transformation $T$, which corresponds to the case where $T = \mathrm{Id}$. Let $(X, \A, \mu)$ be a Lebesgue probability space and take a countably generated sub-$\s$-algebra $\B \subset \A$. In \cite[\S 4]{rohlin_52}, Rokhlin gave a complete description of the possible configurations that arise when we consider such objects. His approach was based on a close study of the conditional measures $(\mu_x)_{x \in X}$ obtained by decomposing $\mu$ over $\B$. In particular, if all of those measures are continuous (i.e. $\forall x, x' \in X$, $\mu_x(\{x'\}) = 0$), then $\B$ has an independent complement: a $\s$-algebra $\C \subset \A$ such that 
\begin{eq}\label{eq:innov}
\C \indep \B \; \text{ and } \; \A = \B \vee \C.
\end{eq}
In the general case where the measures $(\mu_x)_{x \in X}$ have atoms, Rokhlin's description is precise, but written in intricate measure theoretical terms. We prefer the probabilistic formulation found in \cite[Proposition 3.25]{Laurent_standardness-cosiness}: up to embedding $\A$ in a larger $\s$-algebra $\tilde\A$, there is a $\s$-algebra $\C \subset \tilde\A$ such that 
\begin{eq}\label{eq:super-innov}
\C \indep \B \; \text{ and } \; \A \subset \B \vee \C.
\end{eq}
Such a $\C$ is called a \emph{super-innovation}. 

The study of $(X, \A, \mu)$ over $\B$ that we briefly described above is what we refer to as the <<static case>>. Our purpose in this paper is to study the <<dynamical case>>, that arises when a measure preserving transformation $T$ is given and $\B$ is $T$-invariant.

The first question we consider regarding the dynamical case is to compare it to the setup obtained in the static case. We wonder if, in general, there always exists a \emph{dynamical} super-innovation from $\B$ to $\A$, i.e. up to embedding $\bfX$ in a larger system $\tilde\bfX := (\tilde X, \tilde \A, \tilde \mu, \tilde T)$, a $\tilde T$-invariant $\s$-algebra $\C$ satisfying \eqref{eq:super-innov}.

\begin{exe*}
We give here an example to highlight the distinction between (static) super-innovations and dynamical super-innovations. After this, the term <<super-innovation>> will only be used to refer to dynamical super-innovations. 

Take $(\eps_n)_{n \in \bbZ}$ a sequence of independent coin tosses with $1$ or $-1$ on each side of the coin, and $\A$ the associated $\s$-algebra. Consider the $\s$-algebra $\B := \tau^{-1}\A \subset \A$ generated by the cellular automaton  
$$\begin{array}{cccc}
\tau: & \{\pm 1\}^\bbZ & \to & \{\pm 1\}^\bbZ\\
& (\eps_n)_{n \in \bbZ} & \longmapsto & (\eps_n \eps_{n+1})_{n \in \bbZ}\\
\end{array}.$$
From the study done in \cite{article_PL}, we know that the $\s$-algebra generated by $\eps_0$ gives a super-innovation (it is even an independent complement) from $\tau^{-1}\A$ to $\A$. However, if we consider the dynamics given by the shift
$$T: (\eps_n)_{n \in \bbZ} \mapsto (\eps_{n+1})_{n \in \bbZ},$$
then $\tau^{-1}\A$ is a factor $\s$-algebra, but the $\s$-algebra generated by the projection $(\eps_n)_{n \in \bbZ} \mapsto \eps_0$ is not an invariant factor and therefore gives no information about the dynamical structure of $(\{\pm1\}^\bbZ, \A, \mu, T)$ over $\tau^{-1}\A$.

\end{exe*}

We give several examples of extensions with no dynamical super-innovations, which include $\A \to \tau^{-1}\A$, thus showing the first difference with the static case. To get those examples, we introduce the key notion of \emph{confined extensions}: extensions $\A \arr \B$ such that for any joining of $\bfX$ with a system $\bfZ:= (Z, \C, \rho, R)$ such that $\C$ is independent of $\B$, we have that $\C$ is also independent of $\A$(see Definition \ref{def:confined}). This is quite close to the notions of \emph{stability} and \emph{GW-property} presented in \cite{stable_extensions}, but it is easier to use, invariant under isomorphism and applicable in a more general context, without any ergodicity assumptions. Since stability and the GW-property implicitly require that the considered extension be \emph{relatively uniquely ergodic} (see Definition \ref{def:RUE}), comparing confinement to those properties leads us to prove that a confined extension is always isomorphic to a relatively uniquely ergodic extension. We dedicate most of this paper to studying the properties of confined extensions and giving various examples. One property is of particular interest: confined extensions do not admit super-innovations (see Proposition \ref{prop:extension_not_confined}).


As we mentioned earlier, our work on extensions finds an application to the study of \emph{dynamical filtrations}, which are filtrations of the form $\F := (\F_n)_{n \leq 0}$ such that each $\F_n$ is $T$-invariant. The basic setup to study those objects was introduced in \cite{these_PL}, \cite{article_PL}, and we present it in Section \ref{sect:non-standard_I-cosy}. A significant class that arises in this setup is the class of \emph{standard} dynamical filtrations (see Definition \ref{defi:standard_filtration}). This notion of standardness is a translation to the dynamical case of the notion of standardness introduced by Vershik in \cite{Vershik} for a filtration on a probability space. When $T = \mathrm{Id}$, those two notions are equivalent. 

The fact that a dynamical filtration $\F$ is standard imposes some structure on the extensions $\F_{n+1} \arr \F_n$.  To formalize that, we introduce the notion of \emph{standard extension} (see Definition \ref{defi:standard}), which is a weaker property than admitting a super-innovation. The definition of standard extensions is chosen so that, for a standard dynamical filtration $\F$, every extension $\F_{n+1} \arr \F_n$ is standard. Although it is more difficult than finding confined extensions, we also manage to build a non-standard extension, further emphasizing the variety of structures that can arise in the dynamical case. 

In the static case, there are several equivalent criteria to characterize standard filtrations, one of them being I-cosiness. This notion translates to the dynamical case (see \cite[Definition 3.7]{article_PL} or Section \ref{sect:non-standard_I-cosy}). Although it was shown in \cite{article_PL} that standard dynamical filtrations are I-cosy, the converse result was left as an open question. We see in Proposition \ref{prop:non-stadard_I-cosy} that the existence of non-standard extensions gives a negative answer. This is the initial motivation for the work we present in this paper.



\subsubsection*{Outline of the paper}
In Sections \ref{sect:prod_type} and \ref{sect:confined}, we define the main properties that we want to study and we use compact extensions to give concrete examples. In particular, compact extensions give us many examples of confined extensions (see Theorem \ref{thm:confinement_compact}, Proposition \ref{prop:non-WM_compact_ext}), but they are all standard (see Lemma \ref{lem:compact_standard}). In Section \ref{sect:TT-1}, we see that $T, T^{-1}$ transformations give non-compact confined extensions (see Theorem \ref{thm:TT-1_confined}), and we show that, provided $T$ has the so-called PID property, we get a non-standard extension (see Theorem \ref{thm:hyper-confined_extension}). The PID (\emph{pairwise independently determined}) property was introduced by Del Junco and Rudolph in \cite{deljunco_rudolph}, and we recall it in Definition \ref{defi:PID}.

Finally, we give in Section \ref{sect:non-standard_I-cosy} the details of our application of the existence of non-standard extensions to the study of dynamical filtrations.

\subsection{Notations} \label{sect:notations}

A \emph{dynamical system} is a quadruple $\bfX := (X, \A, \mu, T)$ such that $(X, \A, \mu)$ is a Lebesgue probability space, and $T$ is an invertible measure-preserving transformation. We note $\P(X)$ the set of probability measures on $(X, \A)$ and $\P_T(X) \subset \P(X)$ the set of $T$-invariant probability measures. 

Let $\B, \C \subset \A$ be sub-$\sigma$-algebras. We write $\B \subset \C$ mod $\mu$, if for every $B \in \B$, there exists $C \in \C$ such that $\mu(B \Delta C)=0$. Then, $\B = \C$ mod $\mu$ if $\B \subset \C$ mod $\mu$ and $\C \subset \B$ mod $\mu$. We denote $\B \vee \C$ the smallest $\sigma$-algebra that contains $\B$ and $\C$. 
We say that $\B$ is a \emph{factor} $\sigma$-algebra (or $T$-\emph{invariant} $\sigma$-algebra) if $T^{-1}(\B) = \B$ mod $\mu$. Let $\B, \C$ and $\D$ be sub-$\s$-algebras of $\A$. We say that $\B$ and $\C$ are \emph{relatively independent over $\D$} if for any $\B$-measurable bounded function $B$ and $\C$-measurable bounded function $C$
$$\bbE[B C \, | \, \D] = \bbE[B \, | \, \D] \; \bbE[ C \, | \, \D] \; \text{almost surely}.$$
In this case, we write $\B \indep_\D \C$.

If we have two systems $\bfX := (X , \A, \mu, T)$ and $\bfY := (Y, \B, \nu, S)$, a \emph{factor map} is a measurable map $\pi: X \to Y$ such that $\pi_*\mu = \nu$ and $\pi \circ T = S \circ \pi$, $\mu$-almost surely. If such a map exists, we say that $\bfY$ is a \emph{factor} of $\bfX$ and we note $\sigma(\pi) := \pi^{-1}(\B)$ the $\sigma$-algebra generated by $\pi$. Conversely, we also say that $\bfX$ is an extension of $\bfY$ or that $\bfY$ is embedded in $\bfX$. Moreover, if there exist invariant sets $X_0 \subset X$ and $Y_0 \subset Y$ of full measure such that $\pi: X_0 \to Y_0$ is a bijection, then $\pi$ is an \emph{isomorphism} and we write $\bfX \cong \bfY$.

For a given factor $\s$-algebra $\B$, in general, the quadruple $(X, \B, \mu, T)$ is not a dynamical system since $\B$ need not separate points on $X$, and in this case $(X, \B, \mu)$ is not a Lebesgue probability space. However, there exist a dynamical system $\bfY$ and a factor map $\pi: \bfX \to \bfY$ such that $\s(\pi) = \B$ mod $\mu$. Moreover, this representation is not unique, but for a given factor $\B$, there is a canonical construction to get a system $\quotient{\bfX}{\B}$ and a factor map $p_\B: \bfX \to \quotient{\bfX}{\B}$ such that $\s(p_\B) = \B$ mod $\mu$ (see \cite[Chapter 2, Section 2]{glasner}).

A \emph{joining} of $\bfX := (X, \A, \mu, T)$ and $\bfZ := (Z, \C, \rho, R)$ is a $(T \times R)$-invariant measure $\l$ on $X \times Z$ whose marginals are $\mu$ and $\rho$. It yields the dynamical system
$$\bfX \times_\l \bfZ := (X \times Z, \A \otimes \C, \l, T \times R).$$
On this system, the coordinate projections are factor maps that project onto $\bfX$ and $\bfZ$ respectively. If it is not necessary to specify the measure, we will simply write $\bfX \times \bfZ$. For the product joining, we will use the notation $\bfX \otimes \bfZ := \bfX \times_{\mu \otimes \rho} \bfZ$. For the $n$-fold product self-joining, we will write $\bfX^{\otimes n}$. 

Let $\hat\bfX$ be system of which $\bfX$ is a factor, via a factor map $p_X: \hat\bfX \to \bfX$. Any object defined on $\bfX$ has a copy on $\hat\bfX$. For example, for a factor $\s$-algebra $\B$ (resp. a factor map $\pi: \bfX \to \bfY$), we will call a copy of $\B$ (resp. $\pi$) the lift under $p_X$, i.e. $p_X^{-1}(\B)$ (resp. $\pi \circ p_X$). When we have a factor map $\pi: \bfX \to \bfY$, we will also call the factor $\s$-algebra on $\hat\bfX$ generated by $\pi \circ p_X$ a <<copy of $\bfY$>> on $\hat\bfX$. When there is no confusion, we will still denote those copies $\B$, $\pi$, and $\bfY$. Finally, when $\hat\bfX$ is a self-joining of $\bfX$, all objects defined on $\bfX$ will have multiple copies on $\hat\bfX$: in this case we will add a number to identify each copy. For example, on $\bfX^{\times n}$, we will denote $\B_k := p_k^{-1}(\B)$, where $p_k$ is the projection on the $k$-th coordinate.

Assume that $\bfX$ and $\bfZ$ have a common factor, i.e. there are $\B \subset \A$ and $\tilde\B \subset \C$ such that $\quotient{\bfX}{\B} \cong \quotient{\bfZ}{\tilde\B}$. Equivalently, $\bfX$ and $\bfZ$ have a common factor if there are a system $\bfY := (Y, \D, \nu, S)$ and two factor maps $\pi: \bfX \to \bfY$ and $\tilde\pi: \bfZ \to \bfY$. In this case, decompose $\mu$ and $\l$ over $\pi$ and $\pi'$ respectively
$$\mu := \int_Y \mu_y d\nu(y) \; \text{ and } \; \l := \int_Y \l_y d\nu(y).$$
We define the relatively independent product of $\bfX$ and $\bfZ$ over this common factor from the joining
$$\mu \otimes_\bfY \l := \int_Y \mu_y \otimes \l_y d\nu(y).$$
We will denote the resulting system $\bfX \otimes_\bfY \bfZ$ or $\bfX \otimes_{(\B, \tilde\B)} \bfZ$. It has the following well known property (see \cite[Proposition 6.11]{glasner}):
\begin{lem} \label{lem:rel_prod}
Let $\B, \C \subset \A$ be two factor $\s$-algebras. Then $\B$ and $\C$ are independent if and only if, in the relatively independent product of $\bfX$ over $\B$, the two copies of $\C$ are independent.
\end{lem}

\section{Product type, standardness and super-innovations for extensions} \label{sect:prod_type}

Let $\bfX := (X, \A, \mu, T)$ be a dynamical system. We call \emph{extension} a pair of factor $\s$-algebras $\tilde\A$, $\B$ such that $\B \subset \tilde\A$, and we denote it $\tilde\A \arr \B$. To avoid introducing too many notations, we will usually take $\tilde\A = \A$. 

For a given extension $\A \arr \B$, we know that there is a factor map $\pi: \bfX \to \bfY$, unique up to isomorphism, such that $\s(\pi)=\B$ mod $\mu$ (see Section \ref{sect:notations}). For such a factor map, we say that the extension $\A \arr \B$ is given by $\pi$, and we note it $\bfX \overset{\pi}{\to} \bfY$. This representation of extensions is useful in the more concrete cases, but for a general discussion, we find it more convenient to write extensions in terms of $T$-invariant $\s$-algebras.


We first need a notion of isomorphism:
\begin{defi} \label{defi:iso}
Let $\bfX_1 := (X_1, \A_1, \mu_1, T_1)$ and $\bfX_2 := (X_2, \A_2, \mu_2, T_2)$ be dynamical systems. Two extensions $\C \arr \D$ and $\I \arr \J$ on $\bfX_1$ and $\bfX_2$ respectively are \emph{isomorphic} if there exists an isomorphism $\Phi: \quotient{\bfX_1}{\C} \to \quotient{\bfX_2}{\I}$ such that $\Phi \D = \J$ mod $\mu_2$.

In the case where the extensions are given by two factor maps $\pi_1: \bfX_1 \to \bfY_1$ and $\pi_2: \bfX_2 \to \bfY_2$, they are isomorphic if there are two isomorphisms $\phi:  \bfX_1 \to \bfX_2$ and $\psi: \bfY_1 \to \bfY_2$ such that the following diagram is commutative:
$$\xymatrix{
\bfX_1 \ar[r]^{\phi} \ar[d]^{\pi_1} \commutatif & \bfX_2 \ar[d]^{\pi_2}\\
\bfY_1 \ar[r]^{\psi} & \bfY_2\\
}$$
\end{defi}
We write the following definitions in terms of extensions given by factor $\s$-algebras, but they can all be translated for extensions given by a factor map similarly to Definition \ref{defi:iso}.
We then recall the concept of immersion (see \cite{Laurent_standardness-cosiness}), which, in the theory of filtrations, expresses the idea of a <<sub-filtration>>:
\begin{defi}
An extension $\A \arr \B$ is \emph{immersed} in $\C \arr \D$ if we have $\A \subset \C$, $\B \subset \D$ and 
$$\A \indep_{\B} \D.$$
For two extensions $\A \arr \B$ and $\C \arr \D$ defined on two different dynamical systems, we say that $\A \arr \B$ is \emph{immersible} in $\C \arr \D$ if it is isomorphic to an extension immersed in $\C \arr \D$.
\end{defi}

\begin{defi} \label{def:prod_type}
An extension $\A \arr \B$ is of \emph{product type} if there exists a factor $\s$-algebra $\C$ such that $\B \indep \C$ and $\A = \B \vee \C$ mod $\mu$.
\end{defi}

We can finally define standard extensions:
\begin{defi} \label{defi:standard}
An extension is \emph{standard} if it is immersible in a product type extension. More explicitly, an extension $\A \arr \B$ defined on $\bfX := (X, \A, \mu, T)$ is standard if $\bfX$ can be embedded in a system $\hat\bfX$ on which there is an extension $\tilde\B \arr \B$ and a factor $\s$-algebra $\C$ such that $\A \indep_\B \tilde \B$, $\C$ is independent of $\tilde\B$ and $\A \subset \tilde\B \vee \C$.
\end{defi}

For example, we show below that all compact extensions are standard.
\begin{defi} \label{def:compact_ext}
Let $\bfX := (X, \A, \mu, T)$ and $\bfY := (Y, \B, \nu, S)$ be dynamical systems. An extension $\bfX \overset{\pi}{\to} \bfY$ is compact if, up to isomorphism, there exists a compact group $G$, equipped with its Haar measure $m_G$, and a measurable map $\phi: Y \to G$ such that $\bfX = (Y \times G, \nu \otimes m_G, S_\phi)$, where $S_\phi$ is given by
$$S_\phi : (y, g) \longmapsto (Sy, g \cdot \phi(y)).$$
We denote it $\bfY \ltimes_\phi G := \bfX$.
\end{defi}

\begin{lem} \label{lem:compact_standard}
A compact extension is standard.
\end{lem}
\begin{proof}
Let $\bfX := \bfY \ltimes_\phi G$ be a compact extension of $\bfY$. Denote $G_1$ and $G_2$ two copies of $G$ and consider $\bfZ$, the system on $(Y \times G_1 \times G_2, \nu \otimes m_G \otimes m_G)$ given by the transformation
$$(y, g_1, g_2) \mapsto (Sy, g_1 \cdot \phi(y), g_2 \cdot \phi(y)),$$
or, in short, $\bfZ := \bfY \ltimes_{\phi \times \phi} (G_1 \otimes G_2)$. It is isomorphic to the $2$-fold relative product of $\bfX$ over $\bfY$ and can be viewed as in the following diagram

\vspace{5mm}
\centerline{\includegraphics{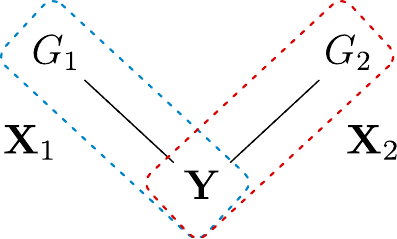}}
For $i= 1, 2$, denote $\bfX_i := \bfY \ltimes_\phi G_i$. We consider the factor map
$$\alpha: (y, g_1, g_2) \mapsto g_1 \cdot {g_2}^{-1},$$
which is independent of the coordinates $(y, g_2)$ that generate $\bfX_2$, because of the invariance of the Haar measure. It is a factor map onto the identity map on $G$ and satisfies 
$$\s(y, g_1, g_2) = \s(y, g_2) \vee \s(\a).$$
This proves that the extension $\bfZ \to \bfX_2$ is of product type. Finally, since the coordinates $y$, $g_1$ and $g_2$ are mutually independent, $\bfX_1$ and $\bfX_2$ are relatively independent over $\bfY$. This means that $\bfX_1 \to \bfY$ is immersed in $\bfZ \to \bfX_2$, and therefore it is standard.
\end{proof}

We also introduce an intermediate property between product type and standardness:
\begin{defi}
An extension $\A \arr \B$ \emph{admits a super-innovation} if there exist a system $\hat\bfX:= (\hat X, \hat \A, \hat\mu, \hat T)$ which extends $\bfX$ such that the extension $\hat\A \arr \B$ is of product type, i.e. there is a factor $\sigma$-algebra $\C$ on $\hat\bfX$ independent of $\B$ such that $\A \subset \hat\A = \B \vee \C$ mod $\hat\mu$.
\end{defi}
An extension that admits a super-innovation is standard because, keeping the notations from the definition, we have that $\A \arr \B$ is immersed in $\hat\A \arr \B$ and $\hat\A \arr \B$ is of product type.
\begin{rmq*} \label{rmq*:super-innov}
For an extension given by a factor map $\pi: \bfX \to \bfY$, we can rewrite the definition of standardness using super-innovations: $\bfX \overset{\pi}{\to} \bfY$ is standard if there exists an extension $\tilde\bfY \overset{\alpha}{\to} \bfY$ such that $\bfX \otimes_\bfY \tilde\bfY \overset{\tilde\pi}{\to} \tilde\bfY$ admits a super-innovation. 

Indeed, if $\bfX \otimes_\bfY \tilde\bfY \overset{\tilde\pi}{\to} \tilde\bfY$ has a super-innovation, we have a system $\bfZ$ such that $\bfX \otimes_\bfY \tilde\bfY \overset{\tilde\pi}{\to} \tilde\bfY$ is immersible in $\bfZ \otimes \tilde\bfY \overset{p}{\to} \tilde\bfY$. Moreover, since, in $\bfX \otimes_\bfY \tilde\bfY$, we have that $\bfX$ is relatively independent of $\tilde\bfY$ over $\bfY$, we get that $\bfX \overset{\pi}{\to} \bfY$ is immersible in $\bfX \otimes_\bfY \tilde\bfY \overset{\tilde\pi}{\to} \tilde\bfY$. Therefore, $\bfX \overset{\pi}{\to} \bfY$ is immersible in $\bfZ \otimes \tilde\bfY \overset{p}{\to} \tilde\bfY$, which means it is standard.

Conversely, if $\bfX \overset{\pi}{\to} \bfY$ is standard, there are two systems $\bfZ$ and $\tilde\bfY$ such that $\bfX \overset{\pi}{\to} \bfY$ is immersible in $\bfZ \otimes \tilde\bfY \overset{p}{\to} \tilde\bfY$. This means we have two factor maps $\a$ and $\beta$ and the following commutative diagram 
$$\xymatrix{
\bfZ \otimes \tilde\bfY \ar[r]^-{p} \ar[d]^{\beta} \commutatif & \tilde \bfY \ar[d]^{\a}\\
\bfX \ar[r]^{\pi} & \bfY\\
},$$
in which $\bfX$ and $\tilde \bfY$ are relatively independent over $\bfY$. Therefore the product map $\beta \times p: Z \times \tilde Y \to X \times \tilde Y$ is a factor map from $\bfZ \otimes \tilde \bfY$ onto $\bfX \otimes_\bfY \tilde\bfY$ which sends $\tilde \bfY$ onto $\tilde\bfY$. This means that $\bfZ$ is a super-innovation for $\bfX \otimes_\bfY \tilde\bfY \overset{\tilde\pi}{\to} \tilde\bfY$.
\end{rmq*}

\begin{prop} \label{prop:super-innov_static}
If $T$ acts as the identity map on $\A$, then $\A \arr \B$ admits a super-innovation.
\end{prop}
\begin{proof}
It follows from \cite[Proposition 3.25]{Laurent_standardness-cosiness}.
\end{proof}

\begin{rmq}
Super-innovations give an intermediate property between product type extensions and standardness. Let us give examples here to show that it is not equivalent to either of these properties. We can sum that up in the following diagram:
$$\begin{array}{ccccc}
\colorbox{lightgray}{Product type} & \begin{array}{c} \Rightarrow \\ \nLeftarrow \\ \end{array} & \colorbox{lightgray}{Admits a super-innovation} &  \begin{array}{c} \Rightarrow  \\ \nLeftarrow \\ \end{array} & \colorbox{lightgray}{Standard}\\
\end{array}$$

\subsubsection*{A standard extension with no super-innovation}
As we have already mentioned, compact extensions are standard, but we will show that, in many cases, they do not admit a super-innovation. For a concrete example, consider the Anzai product given by the map
$$T : (x, y) \mapsto (x + \alpha, y + x), \text{ with } \alpha \in \bbR \backslash \bbQ,$$
on the torus equipped with the Lebesgue measure. It is a compact, and therefore standard, extension of the rotation of angle $\alpha$, but we will see that it has no super-innovation (Proposition \ref{prop:extension_not_confined} and Proposition \ref{prop:non-WM_compact_ext}).

\subsubsection*{An extension which is not of product type but that admits a super-innovation}
There already exist examples in the static case of extensions admitting super-innovations without being of product type, but here we build an \emph{ergodic} example. We will denote $\bfZ := (\{a, b\}, \l, R)$ the two points system on $\{a, b\}$ (we could replace $\bfZ$ by any automorphism with no square root). Then take the disjoint union of two copies $Z_1$ and $Z_2$ of $Z$, and consider the transformation $T$ which sends $Z_1$ onto $Z_2$ via $R$ and sends $Z_2$ onto $Z_1$ via the identity map. Formally, denote $\bfY := (\{1, 2\}, \nu, S)$ the two points system on $\{1, 2\}$, and define $\bfX$ to be the system on the product space $(\{1, 2\} \times \{a, b\}, \nu \otimes \l)$ given by the map
$$T(i, z) := \left\{
\begin{array}{cl}
(2, Rz) & \text{ if } i = 1\\
(1, z) & \text{ if } i = 2\\
\end{array} \right..
$$
As we see on the following diagram, $\bfX$ is simply a cyclic four points system on $\{1, 2\} \times \{a, b\}$:

\centerline{\includegraphics[scale=1.3]{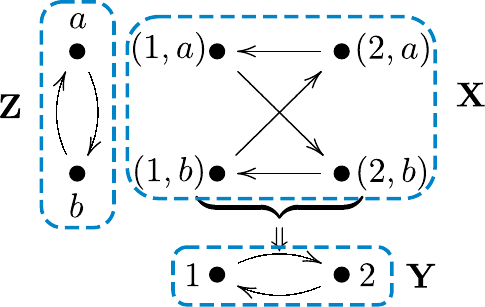}}

If $\A$ is the full $\sigma$-algebra on $\bfX$ and $\B$ the factor generated by the projection on $\{1, 2\}$, the extension $\A \arr \B$ admits a super-innovation. Indeed, consider the transformation $\hat R: (z_1, z_2) \mapsto (z_2, Rz_1)$ and the system $\hat\bfZ := (Z \times Z, \l \otimes \l, \hat R)$. Finally, define $\hat\bfX$ as the direct product of the two points system on $\{1, 2\}$, i.e. $\bfY$, and $\hat\bfZ$, which extends $\bfX$ via the factor map
$$\begin{array}{cccc}
\pi: & \hat\bfX & \to & \bfX\\
& (i, z_1, z_2) & \longmapsto & (i, z_i)\\
\end{array}.$$
An orbit on $\hat\bfX$ goes as follows
$$\begin{array}{c}
\underbrace{\xymatrix{
{\left(\begin{array}{c} a \\ b \\ 1 \\ \end{array}\right)} \ar@{|->}[r]^{S \times \hat R} & {\left(\begin{array}{c} b \\ b \\ 2 \\ \end{array}\right)} \ar@{|->}[r]^{S \times \hat R} & {\left(\begin{array}{c} b \\ a \\ 1 \\ \end{array}\right)} \ar@{|->}[r]^{S \times \hat R} & {\left(\begin{array}{c} a \\ a \\ 2 \\ \end{array}\right)} \ar@{|->}[r]^{S \times \hat R} & {\left(\begin{array}{c} a \\ b \\ 1 \\ \end{array}\right)}
}}\\
\includegraphics[scale=0.4]{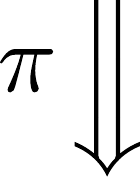} \hspace{0.33cm} \\
\xymatrix{
{\left(\begin{array}{c} a \\ 1 \\ \end{array}\right)} \ar@{|->}[r]^{T} & {\left(\begin{array}{c} b \\ 2 \\ \end{array}\right)} \ar@{|->}[r]^{T} & {\left(\begin{array}{c} b \\ 1 \\ \end{array}\right)} \ar@{|->}[r]^{T} & {\left(\begin{array}{c} a \\ 2 \\ \end{array}\right)} \ar@{|->}[r]^{T} & {\left(\begin{array}{c} a \\ 1 \\ \end{array}\right)}
}\\
\end{array}$$
It is then clear that $\hat\bfZ$ gives us the desired super-innovation.

However, the extension $\A \arr \B$ is not of product type. Indeed, if it were, there would exist a system $\bfW := (W, \gamma, Q)$ and an isomorphism $\Phi: \bfY \otimes \bfW \to \bfX$ which sends $\bfY$ onto $\bfY$. By that, we mean that there would exist two measure preserving bijections $\phi_i: W \to Z$ for $i = 1 , 2$ such that 
$$\Phi(i, w) = (i, \phi_i(w)) \text{ almost surely}.$$
Then, the identity $\Phi \circ (S \times Q) = T \circ \Phi$ would become:
$$\phi_2^{-1} \circ R \circ \phi_1 = Q = \phi_1^{-1} \circ \phi_2.$$
This would give 
$$R = \phi_2 \circ \phi_1^{-1} \circ \phi_2 \circ \phi_1^{-1} = \phi \circ \phi,$$
with $\phi := \phi_2 \circ \phi_1^{-1}$. Since $\phi \in \aut(Z, \l)$, this would contradict the fact that $R$ has no square root.
\end{rmq}

\section{Confined extensions} \label{sect:confined}

In trying to build non-standard extensions, we first look for extensions with no super-innovations. The notion of \emph{confined extension} that we introduce in this section follows that purpose, relying on the non-trivial joining properties associated to super-innovations. The link between confined extensions and super-innovations will be detailed in Section \ref{sect:confined_no_super-innov}.

\subsection{Definitions, basic properties and examples}

\begin{defi} \label{def:confined}
Let $\bfX := (X, \A, \mu, T)$ be a dynamical system and $\B$ be a factor $\s$-algebra. The extension $\A \arr \B$ is said to be \emph{confined} if it satisfies one of the following equivalent properties:
\begin{enumerate}[label = (\roman*)]
\item every $2$-fold self-joining of $\bfX$ in which the two copies of $\B$ are independent is the product joining;
\item for every system $\bfZ$ and every joining of $\bfX$ and $\bfZ$ in which the copies of $\B$ and $\bfZ$ are independent, the copies of $\A$ and $\bfZ$ are independent;
\item for every $n \in \bbN^* \cup \{+ \infty\}$, every $n$-fold self-joining of $\bfX$ in which the $n$ copies of $\B$ are mutually independent is the $n$-fold product joining.
\end{enumerate}
\end{defi}

We will prove the equivalence in Proposition \ref{prop:alt_def_confined}. Let us present here some examples of confined extensions:

\begin{description} [leftmargin=*]
\item[Compact extensions.] This is the most well known family of extensions, and it is therefore natural to start our study with them. We have a criterion for the confinement of compact extensions, which we state below. We thank Mariusz Lema\'ncyzk  for suggesting this criterion. Let $\bfX := (X, \A, \mu, T)$, $\bfY := (Y, \B, \nu, S)$ and $\bfX \overset{\pi}{\to} \bfY$ be a compact extension, i.e. $\bfX = \bfY \ltimes_\phi G$, using notations from Definition \ref{def:compact_ext}. Consider the ergodic decomposition of $\nu \otimes \nu$:
$$\nu \otimes \nu := \int \rho_\o \, d \, \bbP(\o).$$
This gives a decomposition (not necessarily ergodic) of $\nu \otimes \nu \otimes m_G \otimes m_G$:
$$\nu \otimes \nu \otimes m_G \otimes m_G = \int \rho_\o \otimes m_G \otimes m_G \, d \, \bbP(\o).$$
We can switch the coordinates from $(Y \times G) \times (Y \times G)$ to $Y \times Y \times G \times G$ so that $S_\phi \times S_\phi$ become 
$$(S \times S)_{\phi \times \phi}: (y_1, y_2, g_1, g_2) \mapsto (Sy_1, Sy_2, g_1 \cdot \phi(y_1), g_2 \cdot \phi(y_2)).$$
Each measure $\rho_\o \otimes m_G \otimes m_G$ is invariant under $(S \times S)_{\phi \times \phi}$.
\begin{thm} \label{thm:confinement_compact}
The following are equivalent:
\begin{enumerate} [label= (\roman*)]
\item The extension $\bfX \overset{\pi}{\to} \bfY$ is confined;
\item The product extension $\bfX \otimes \bfX \overset{\pi \times \pi}{\to} \bfY \otimes \bfY$ is relatively ergodic;
\item For $\bbP$-almost every $\o$ the measure $\rho_\o \otimes m_G \otimes m_G$ is ergodic under $(S \times S)_{\phi \times \phi}$.
\end{enumerate}
In particular, weakly mixing compact extensions are confined.
\end{thm}

In Section \ref{sect:stable_extensions}, we will study the link between confinement and Robinson's notion of \emph{stable extensions}. We can deduce the weakly mixing case of Theorem \ref{thm:confinement_compact} from Robinson's work by combining \cite[Corollary 3.8]{stable_extensions} and Lemma \ref{lem:stable-confined} (i). See Section \ref{sect:compact} for a full proof of the theorem. 
Our theorem is easy to use for weakly mixing compact extensions, but for the non-weakly mixing case, our condition is more involved. In Section \ref{sect:compact}, we give an application in the non-weakly mixing case with an Anzai skew-product (see Proposition \ref{prop:non-WM_compact_ext}). We also give an example of a non-confined ergodic compact extension, illustrating that some condition is still necessary for compact extensions to be confined.


\item[$T, T^{-1}$ transformations.] Using arguments from Lema\'nczyk and Lesigne \cite{erg_rokhlin_cocycles}, we show that, provided $T^2$ is ergodic, $T, T^{-1}$ transformations yield confined extensions (see Theorem \ref{thm:TT-1_confined}). Moreover, if $T$ is weakly mixing, the $T, T^{-1}$ extension is confined and not compact (see Corollary \ref{cor:TT-1_relWM}). Finally we show that, with more assumptions on $T$, we get an additional property: it is not \emph{standard} (see Theorem \ref{thm:hyper-confined_extension}).

\item[Flow extensions.] A generic flow extension of a weakly mixing system is confined (see \cite[Section 8]{stable_extensions} for the definitions). We do not discuss those examples in detail here, but Robinson showed in \cite{stable_extensions} that such extensions have the \emph{GW-property} and we show below that this property yields confinement (see Proposition \ref{prop:GW_confined}).

\item[Totally confined systems.] We can build systems with a surprising property: they are a confined extension of any non-trivial factor. We call them \emph{totally confined systems}. The existence of such systems with non-trivial factors follows from the work done in \cite{JP_systems}: a system verifying the JP-property is totally confined. Building on similar arguments, it can be shown that any system whose reduced maximal spectral type verifies that $\s * \s$ is disjoint from $\s$ is totally confined.
\end{description}

The key proposition is the following, where we prove the equivalence from Definition \ref{def:confined}:
\begin{prop} \label{prop:alt_def_confined}
Properties (i), (ii) and (iii) in Definition \ref{def:confined} are equivalent. Therefore, either of those properties can be used as a definition of confined extensions.
\end{prop}
\begin{proof}
\begin{description} [leftmargin=*]
\item[(i) $\Rightarrow$ (ii).]
Assume that the extension $\A \arr \B$ satisfies property (i). Let $\bfZ$ be a dynamical system and $\hat\bfX := \bfX \times \bfZ$ a joining of $\bfX$ and $\bfZ$ for which $\B$ and $\bfZ$ are independent. We then look at the system $\tilde\bfX$ built as the relatively independent product of $\hat\bfX$ over $\bfZ$. Because of our assumption on $\hat\bfX$ and Lemma \ref{lem:rel_prod}, the copies of $\B$ in $\tilde\bfX$ are independent. Therefore, by property (i), the copies of $\A$ are independent. However, using again Lemma \ref{lem:rel_prod}, this is only possible if $\A$ is independent of $\bfZ$ in $\hat\bfX$. And this gives us property (ii).
\item[(ii) $\Rightarrow$ (iii).]
Assume that $\A \arr \B$ satisfies property (ii). Let $n \in \bbN$ (if $n = +\infty$, we only need to show that finite families of copies of $\A$ are mutually independent) and let $\bfZ := \bfX \times \cdots \times \bfX$ be a $n$-fold joining of $\bfX$ for which the copies of $\B$ are independent. We show by induction on $k$ that the family $(\A_1, ..., \A_k,  \B_{k+1}, ..., \B_n)$ is mutually independent. The case $k = 0$ is simply our assumption on $\bfZ$. If the property is true for $k$, then $\B_{k+1}$ is independent of $(\A_1, ..., \A_k,  \B_{k+2}, ..., \B_n)$, therefore, using (ii), we get that $\A_{k+1}$ is independent of $(\A_1, ..., \A_k, \B_{k+2}, ..., \B_n)$. Since the family $(\A_1, ..., \A_k, \B_{k+2}, ..., \B_n)$ is mutually independent, it implies that $(\A_1, ..., \A_k, \A_{k+1}, \B_{k+2}, ..., \B_n)$ is mutually independent. The case $k=n$ ends our proof.
\item[(iii) $\Rightarrow$ (i).] Simply take $n = 2$.
\end{description}
\end{proof}

Let us state some simple manipulations possible with confined extensions:
\begin{prop} \label{prop:composition_confined_extensions}
Let $\C \subset \B \subset \A$ be three factor $\s$-algebras. The extension $\A \arr \C$ is confined if and only if the extensions $\A \arr \B$ and $\B \arr \C$ are both confined. 
\end{prop}
\begin{proof}
This follows from the definitions.
\end{proof}

\begin{prop} \label{prop:independent_extension}
Let $\bfX$ be a dynamical system, $\A \arr \B$ be a confined extension on $\bfX$ and let $\C$ be a factor $\s$-algebra independent of $\A$. Then the extension $\A \vee \C \arr \B \vee \C$ is confined.
\end{prop}
\begin{proof}
Let $\bfZ$ be a dynamical system and take a joining with $\bfX$ in which $\B \vee \C$ is independent of $\bfZ$. Therefore, $\B$, $\C$ and $\bfZ$ are mutually independent, so $\B$ is independent of $\C \vee \bfZ$. Using the definition of confined extensions, this implies that $\A$ is independent of $\C \vee \bfZ$. This means that $\A \vee \C$ is independent of $\bfZ$, which shows that the extension $\A \vee \C \arr \B \vee \C$ is confined.
\end{proof}

\begin{rmq} \label{rmq:attention}
Let us however note that some manipulations on confined extensions are not always true: 
\begin{enumerate} [label= (\roman*)]
\item For two confined extensions $\A \arr \B$ and $\tilde\A \arr \B$ over a same factor, we cannot conclude that the joint extension $\A \vee \tilde\A \arr \B$ is confined. Therefore, in this case, there is not a \guillemotleft{} largest confined extension of $\B$ \guillemotright, since such an extension would contain $\A \vee \tilde\A$. A counter-example, which relies on compact extensions, can be found in Example \ref{exe:sup_non_confinée}.
\item Our first remark implies that the independence condition in Proposition \ref{prop:independent_extension} cannot be removed. Indeed, take $\A$, $\tilde\A$ and $\B$ such that $\A \arr \B$ and $\tilde\A \arr \B$ are confined but $\A \vee \tilde \A \arr \B$ is not. We have
$$\A \vee \tilde\A \arr \B \vee \tilde \A = \tilde \A \arr \B.$$
Therefore, using Proposition \ref{prop:composition_confined_extensions}, $\A \vee \tilde\A \arr \B \vee \tilde \A$ is not confined, even though $\A \arr \B$ is.
\end{enumerate}
\end{rmq}

\subsection{Confined extensions do not admit a super-innovation} \label{sect:confined_no_super-innov}

We mentioned in the previous section that our initial interest in confined extensions stemmed from the fact that they do not admit super-innovations. Let us prove this here. We need the following (and basic) result:
\begin{lem} \label{lem:independance_mesurabilite}
Let $(X, \mu)$ be a probability space. Let $\A$, $\B$ and $\C$ be $\s$-algebras such that $\B \subset \A$, $\A \subset \B \vee \C$ mod $\mu$ and $\C$ is independent of $\A$. Then $\A = \B$ mod $\mu$.
\end{lem}
\begin{proof}
Let $A$ be a bounded real-valued $\A$-measurable random variable. Since $\A \subset \B \vee \C$, we have 
$$A = \bbE[A \, | \, \B \vee \C].$$
To identify the right-hand term, take a $\B$-measurable bounded random variable $B$ and a $\C$-measurable bounded random variable $C$. We know that $A B$ and $C$ are independent, so
\begin{align*}
\bbE[ A B C ] & = \bbE[A B] \; \bbE[C]\\
&= \bbE\big[ \bbE[A \, | \, \B] \, B \big] \;  \bbE[C]\\
&= \bbE\big[ \bbE[A \, | \, \B] \, B C\big].
\end{align*}
So 
$$A = \bbE[A \, | \, \B \vee \C] = \bbE[A \, | \, \B],$$
which means that $A$ is $\B$-measurable. Therefore, $\A \subset \B$ mod $\mu$.
\end{proof}
 
We can now show that super-innovations and confined extensions are incompatible:
\begin{prop} \label{prop:extension_not_confined}
%
Assume that $\B$ is a proper factor $\s$-algebra of $\bfX := (X, \A, \mu, T)$ (i.e. we assume that $\B \neq \A$) and that the extension $\A \arr \B$ admits a super-innovation. Then $\A \arr \B$ is not confined.
\end{prop}

\begin{proof}
Assume that the extension $\A \arr \B$ is confined. Since it admits a super-innovation, there exist a system $\bfZ$ and a joining $\bfX \times_\l \bfZ$ in which we have $\A \subset \B \vee \bfZ$ mod $\l$ and $\B$ and $\bfZ$ are independent. However, since $\A \arr \B$ is confined, this means that $\A$ and $\bfZ$ are independent. Now using Lemma \ref{lem:independance_mesurabilite}, we get that $\A = \B$ mod $\mu$, which is contrary to our assumption.

\end{proof}
Combining this with Proposition \ref{prop:composition_confined_extensions}, we get the following corollary, which is very useful when we want to show that an extension is not confined.
\begin{cor} \label{cor:extension_not_confined}
Let $\bfX:= (X, \A, \mu, T)$ be a dynamical system. If the extension $\A \arr \B$ is confined, then for any factor $\s$-algebra $\tilde \A$ such that $\B \subsetneq \tilde\A \subsetneq \A$ mod $\mu$, neither $\A \arr \tilde\A$ nor $\tilde\A \arr \B$ admit a super-innovation. In particular, there cannot be a factor in $\A$ independent of $\B$.
\end{cor}
\begin{exe} \label{exe:sup_non_confinée}
Let us use our corollary to illustrate Remark \ref{rmq:attention}. Take a system $\bfY$, a measurable map $\phi: \bfY \to G$ and consider the system $\bfY \ltimes_\phi G$ as defined in Definition \ref{def:compact_ext}. Because of Theorem \ref{thm:confinement_compact}, we can choose $\bfY \ltimes_\phi G$ so that the resulting compact extension is confined. As in the proof of Lemma \ref{lem:compact_standard}, take $G_1$ and $G_2$ two copies of $G$ and consider $\bfZ$, the system on $(Y \times G_1 \times G_2, \nu \otimes m_G \otimes m_G)$ given by the transformation
$$(y, g_1, g_2) \mapsto (Sy, g_1 \cdot \phi(y), g_2 \cdot \phi(y)). $$
In this case, $\bfZ \to \bfY$ is the supremum of the compact extensions $\bfY \ltimes_\phi G_1 \to \bfY$ and $\bfY \ltimes_\phi G_2 \to \bfY$. However, we have the invariant map 
$$\a : (y, g_1, g_2) \mapsto g_2 \cdot g_1^{-1}$$
which is independent from $\bfY$. Therefore, Corollary \ref{cor:extension_not_confined} tells us that $\bfZ \to \bfY$ is not confined. 

This gives an example of a supremum of two confined extensions which is not confined.
\end{exe}

\subsection{Lifting results}

In this section we list some properties of a dynamical system which are automatically lifted to any confined extension. Such results are not surprising. Indeed, we show in the next section that confined extensions resemble Robinson's \emph{stable extensions} which he developed specifically to get lifting results \cite{stable_extensions}.

\begin{prop} \label{prop:lifting}
Let be $\calP$ be a property of a system and let $\C$ be a family of dynamical systems. Assume that $\calP$ can be characterized as follows: a system $\bfX$ satisfies $\calP$ if and only if, for every $\bfZ \in \C$, $\bfX$ is disjoint from $\bfZ$. If $\bfY$ satisfies $\calP$ and the extension $\bfX \overset{\pi}{\to} \bfY$ is confined, then $\bfX$ satisfies $\calP$.
\end{prop}
\begin{proof}
It follows directly from Definition \ref{def:confined} (ii).
\end{proof}

Using this proposition, we can prove that many properties are preserved under confined extensions:
\begin{enumerate}
\item Ergodicity: $\bfX$ is ergodic if and only if it is disjoint from every identity system (see \cite[Theorem 6.26]{glasner}).
\item Weak mixing: $\bfX$ is weakly mixing if and only if it is disjoint from every system with discrete spectrum (see \cite[Theorem 6.27]{glasner}).
\item Mild mixing: $\bfX$ is mildly mixing if and only if it is disjoint from every rigid system (see \cite[Corollary 8.16]{glasner}).
\item $K$-property: $\bfX$ is a $K$-automorphism if and only if it is disjoint from every $0$-entropy system (see \cite[Theorem 18.16]{glasner}).
\end{enumerate}
\begin{rmq} \label{rmq:disjoint_ergodic}
Conversely, one can use confinement to characterize disjointness from a family of systems. For example, the systems disjoint from all ergodic systems are the confined extensions of identity map systems (see \cite{disjoint_from_erg}).
\end{rmq}

We can also prove that other properties are preserved under confined extensions. In the following proposition, we prove this to be true for mixing of all orders:
\begin{defi}
Let $n \geq 2$. A system $\bfX$ is $n$-mixing if for all measurable sets $A_1, ..., A_n \subset X$ we have
$$\lim_{k_1, ..., k_{n-1} \rightarrow \infty} \mu(A_1 \cap T^{-k_1} A_2 \cap \cdots \cap T^{-(k_1 + \cdots + k_{n-1})}A_n) = \mu(A_1) \cdots \mu(A_n).$$
\end{defi}
\begin{prop}
Let $n \geq 2$. If $\bfY$ is $n$-mixing and the extension $\bfX \overset{\pi}{\to} \bfY$ is confined, then $\bfX$ is $n$-mixing.
\end{prop}
\begin{proof}
Let $J_n(\mu) \subset \P(X^{n})$ be the set of $n$-fold joinings of $\mu$. We endow it with the topology given by: $\l_p \underset{p \arr \infty}{\to} \l$ if
$$\forall A_1, ..., A_n \in \A, \l_p(A_1 \times \cdots \times A_n) \underset{p \arr \infty}{\to} \l(A_1 \times \cdots \times A_n).$$
With this topology, $J_n(\mu)$ is a compact metrizable space (see \cite{joinings_in_ergodic_theory}). 

For $k:= (k_1, ..., k_{n-1}) \in \bbN^{n-1}$, let us define the measure on $X^n$:
$$\mu_k := \int \delta_x \otimes \delta_{T^{k_1}x} \otimes \cdots \otimes \delta_{T^{k_1 + \cdots k_{n-1}}x} d\mu(x),$$ 
and similarly define $\nu_k$ on $Y^n$. We can re-write the definition of $n$-mixing as 
$$\mu_k \underset{k \rightarrow \infty}{\to} \mu^{\otimes n}.$$
We also define $\pi_n := \pi \times \cdots \times \pi$ and $T_n := T \times \cdots \times T$. 

Using the compactness of $J_n(\mu)$, fix a sequence $\left(k(i)\right)_{i \in \bbN}$ on $\bbN^{n-1}$ such that, for every $\ell \in \{1, ..., n-1\}$ 
$$\lim_{i \rightarrow \infty} k_\ell(i) = + \infty,$$
and 
$$\mu_{k(i)} \underset{i \rightarrow \infty}{\to} \l,$$
for some measure $\l$ on $X^n$. Clearly, $\l$ is a $T_n$-invariant measure which projects to $\mu$ on each coordinate, and that means that it defines a $n$-fold joining of $\bfX$. Moreover, we have $(\pi_n)_*\mu_k = \nu_k$, which yields
$$(\pi_n)_*\l = \lim_{i \rightarrow \infty} \nu_{k(i)} = \nu^{\otimes n},$$
where we use the $n$-mixing property of $\bfY$ to get the last equality. Finally, since the extension $\bfX \overset{\pi}{\to} \bfY$ is confined, using property (iii) from Definition \ref{def:confined}, we must have $\l = \mu^{\otimes n}$.

We have proved that $\mu_k \underset{k \rightarrow \infty}{\to} \mu^{\otimes n}$, which means that $\bfX$ is $n$-mixing.
\end{proof}

Using again property (iii) from Definition \ref{def:confined}, we can easily prove that Del Junco and Rudolph's PID property is preserved under confined extensions. We first recall the definition:
\begin{defi} [Del Junco and Rudolph \cite{deljunco_rudolph}] \label{defi:PID}
Let $\bfX$ be a dynamical system and $n \in \bbN \cup \{+\infty\}$. We say that $\bfX$ has the $n$-fold PID (pairwise independently determined) property if the only $n$-fold self-joining of $\bfX$ in which the copies of $\bfX$ are pairwise independent is the product joining.
\end{defi}
We then have:
\begin{prop}
If $\bfY$ has the $n$-fold PID property and the extension $\pi: \bfX \to \bfY$ is confined, then $\bfX$ has the $n$-fold PID property.
\end{prop}
\begin{proof}
It follows from property (iii) in Definition \ref{def:confined}.
\end{proof}

We can also see that confined extensions preserve the Kolmogorov-Sinaï entropy:
\begin{prop}
If $\bfX \overset{\pi}{\to} \bfY$ is a confined extension, then $h(\bfX) = h(\bfY)$. Moreover, if $\bfX$ is a $K$-automorphism of finite entropy, the converse is true: $\bfX \overset{\pi}{\to} \bfY$ is confined if and only if $h(\bfX) = h(\bfY)$.
\end{prop}
The second part of the proposition was pointed out to us by Christophe Leuridan.
\begin{proof}
Assume that $h(\bfX) > h(\bfY)$. Using Thouvenot's relative version of Sinaï's theorem, we know there exists a Bernoulli factor of $\bfX$ with entropy $h(\bfX) - h(\bfY)$ which is independent of $\bfY$ (we can get that result from \cite[Proposition 2]{thouvenot_ornstein_rel}). But we have seen in Corollary \ref{cor:extension_not_confined} that a confined extension of $\bfY$ can have no non-trivial factors independent of $\bfY$, therefore $\bfX \overset{\pi}{\to} \bfY$ is not confined. This proves the first part of the proposition.

To prove the second part, we will use \cite[Lemma 2]{K_confined}. It gives a relative version of the disjointness of $K$-automorphisms and 0-entropy systems: On a dynamical system $(Z, \C, \rho, R)$, take two $R$-invariant $\s$-algebras $\A$ and $\B$ such that $\B \subset \A$ and $h(\A, R) = h(\B, R) < \infty$. Next, take a third $R$-invariant $\s$-algebra $\D$ with finite entropy such that $(\D, R)$ has the $K$-property and $\D$ is independent of $\B$. Then, $\D$ is independent of $\A$.

Assume that $\bfX$ is a $K$-automorphism of finite entropy and that $h(\bfX) = h(\bfY)$. Let $\l$ yield a self-joining $\bfX_1 \times_\l \bfX_2$ in which the copies of $\bfY$ are independent. First, $\bfY_2$ has the $K$-property and is independent of $\bfY_1$. Moreover, $h(\bfX_1) = h(\bfY_1)$, so \cite[Lemma 2]{K_confined} tells us that $\bfY_2$ is independent of $\bfX_1$. Similarly, $\bfX_1$ has the $K$-property and is independent of $\bfY_2$. Since $h(\bfX_2) = h(\bfY_2)$, \cite[Lemma 2]{K_confined} tells us that $\bfX_1$is independent of $\bfX_2$.
\end{proof}

However, not all properties are preserved under confined extensions:
\begin{enumerate}
\item Rigidity: the system studied in Proposition \ref{prop:non-WM_compact_ext} and Proposition \ref{prop:non_rigid_confined_extension} gives an example of a non-rigid confined extension of a rigid factor. We recall the definition of rigidity in Definition \ref{def:rigid}.
\item The Bernoulli and loosely Bernoulli properties: we show in Theorem \ref{thm:TT-1_confined} that a $T, T^{-1}$ transformation is a confined extension of its natural Bernoulli factor (provided $T^2$ is ergodic), but, when $T$ is given by a Bernoulli shift, Kalikow proved in \cite{kalikow_TT-1} that the $T, T^{-1}$ transformation is not loosely Bernoulli (and therefore not Bernoulli either).
\end{enumerate}

\subsection{Confinement, stability and GW-property}

In studying confined extensions, we found in the literature similar notions : stability and GW-property. The purpose of this section is to compare confinement to those properties.

A key notion of this section will be the following. It already appeared implicitly in Lemma \ref{lem:rel_unique_ergo}, but we choose to formalize it here.
\begin{defi} \label{def:RUE}
Let $\bfX := (X, \A, \mu, T)$ and $\bfY := (Y, \B, \nu, S)$. We say that the extension given by a factor map $\pi: \bfX \to \bfY$ is \emph{relatively uniquely ergodic (RUE)} over $\nu$ if, for any $T$-invariant probability measure $\l$ on $X$ such that $\pi_*\l = \nu$, we have $\l = \mu$.
\end{defi}

\subsubsection{Stable extensions} \label{sect:stable_extensions}

In \cite{stable_extensions}, Robinson gives three notions of << stable extensions >>, which we recall here.
\begin{defi} \label{defi:stable}
Let $\bfX := (X, \A, \mu, T)$ and $\bfY := (Y, \B, \nu, S)$ be dynamical systems. Consider an extension given by a factor map $\pi: \bfX \to \bfY$. We say that $\bfX \overset{\pi}{\to} \bfY$ is:
\begin{itemize}
\item \emph{stable} if $\bfX$ is ergodic and for every system $\bfZ := (Z, \C, \rho, R)$ such that $\bfX \otimes \bfZ$ is ergodic, the extension $\bfX \otimes \bfZ \overset{\pi \times \mathrm{Id}}{\to} \bfY \otimes \bfZ$ is relatively uniquely ergodic over $\nu \otimes \rho$. 
\item \emph{$n$-fold self-stable} if $\bfY$ is weakly mixing and the extension $\bfX^{\otimes n} \overset{\pi_n}{\to} \bfY^{\otimes n}$ is relatively uniquely ergodic over $\nu^{\otimes n}$. It is \emph{self-stable} if it is $n$-fold self-stable, for every $n$. 
\item \emph{weakly stable} if $\bfX$ is ergodic and for every system $\bfZ$ and every \emph{ergodic} joining of $\bfX$ and $\bfZ$ for which $\bfY$ and $\bfZ$ are independent, $\bfX$ and $\bfZ$ are also independent. 
\end{itemize}
\end{defi}
\begin{rmqs}
The main difference between weak stability and the first two definitions is that weak stability is an isomorphism invariant (see \cite[Proposition 3.12]{stable_extensions}), while stability and self-stability depend on the model we consider. 

Then, the distinction between stability and confinement lies mainly in the ergodicity and weak mixing assumptions in the definitions of stability. We discuss this more precisely in the next proposition.
\end{rmqs}

\begin{prop} \label{lem:stable-confined}
Consider an extension given by a factor map $\pi: \bfX \to \bfY$ and let $n \geq 1$. We have the following relations:
\begin{enumerate} [label = (\roman*)]
\item $\bfX \overset{\pi}{\to} \bfY$ is $n$-fold self-stable if and only if $\bfY$ is weakly mixing, and $\bfX \overset{\pi}{\to} \bfY$ is RUE and confined,
\item if $\bfY$ is ergodic and $\bfX \overset{\pi}{\to} \bfY$ is confined, then $\bfX \overset{\pi}{\to} \bfY$ is weakly stable.
\suspend{enumerate}
Since we have \cite[Proposition 3.13.]{stable_extensions}, we also get
\resume{enumerate}[{[label = (\roman*)]}]
\item If $\bfY$ is ergodic and $\bfX \overset{\pi}{\to} \bfY$ is RUE and confined, then $\bfX \overset{\pi}{\to} \bfY$ is stable.
\end{enumerate}
\end{prop}
Using those results, we can answer some questions left open by Robinson (see \cite[\S 3.3]{stable_extensions}) on stable extensions: 
\begin{enumerate}
\item $2$-fold self-stability implies self-stability,
\item and self-stability implies stability. 
\end{enumerate}
Proving those result does not require the use of confined extensions, we could also have proven them directly using similar arguments to those used in Proposition \ref{prop:alt_def_confined}.

\begin{proof}[Proof of Proposition \ref{lem:stable-confined}]
In this proof we say that an extension satisfying Definition \ref{def:confined} (iii) is $n$-confined.

Let us prove (i): assume that $\bfX \overset{\pi}{\to} \bfY$ is $n$-fold self-stable. By definition, $\bfY$ is weakly mixing. Then, \cite[Lemma 3.7, (v)]{stable_extensions} gives the relative unique ergodicity. Finally, the $n$-fold self stability tells us that the only $T^n$-invariant measure which projects to $\nu^{\otimes n}$ is the product measure $\mu^{\otimes n}$, which means the extension $\bfX \overset{\pi}{\to} \bfY$ is $n$-confined, and therefore confined. Conversely, assume that $\bfY$ is weakly mixing and that $\bfX \overset{\pi}{\to} \bfY$ is confined and RUE. We then know that $\bfX \overset{\pi}{\to} \bfY$ is $n$-confined. To prove the $n$-fold self-stability: let $\l \in \P(X^n)$ be $T^n$-invariant and assume it projects to $\nu^{\otimes n}$. Using the RUE property of $\bfX \overset{\pi}{\to} \bfY$, we get that $\l$ is a $n$-fold joining of $\mu$. Then the $n$-confinement of $\bfX \overset{\pi}{\to} \bfY$ implies that it is the product joining, that is $\l = \mu^{\otimes n}$. 

We then prove (ii). Since $\bfY$ is ergodic and $\bfX \overset{\pi}{\to} \bfY$ confined, using Proposition \ref{prop:lifting}, we get that $\bfX$ is ergodic. Then (ii) follows from the definitions.

We get (iii) by combining (ii) and \cite[Proposition 3.13]{stable_extensions}.
%
%
\end{proof}

\subsubsection{GW-property}

The GW-property was introduced by Glasner and Weiss in \cite{GW_paper} and named so by Robinson in \cite{stable_extensions}. Robinson defines this property on topological models: that is, $X$ and $Y$ are compact metric spaces and $T$, $S$ and $\pi$ are continuous maps. However, it will be more convenient here to define it in the more general setup of standard Borel spaces. 

Let $\bfX := (X, \A, \mu, T)$, $\bfY := (Y, \B, \nu, S)$ where, as in the rest of this paper, $(X, \A)$ and $(Y, \B)$ are standard Borel spaces. Let $\pi: \bfX \to \bfY$ be a measurable factor map, that is, $\pi: (X, \A) \to (Y, \B)$ is a Borel map. We define $\M_\nu$ as the set of probability measures on $X$ which project to $\nu$ under $\pi$:
$$\M_\nu := \{\g \in \P(X) \, | \, \pi_*\g = \nu\}.$$
Since $\P(X)$ equipped with the weak* topology is a Polish space, $\M_\nu$ is a standard Borel space. Moreover, it is $T_*$-invariant, and therefore we can consider the measurable action of $T_*$ on $\M_\nu$. Note that $\mu$ is a fixed point for $T_*$, and therefore $\delta_\mu$ is a $T_*$-invariant measure on $\M_\mu$.
\begin{defi}
Using the notations introduced above, we say that $\bfX \overset{\pi}{\to} \bfY$ has the GW-property if $(\M_\nu, \delta_\mu, T_*)$ is uniquely ergodic, i.e. $\delta_\mu$ is the only $T_*$-invariant measure on $\M_\nu$.
\end{defi}
\begin{prop} \label{prop:GW_confined}
The extension $\bfX \overset{\pi}{\to} \bfY$ has the GW-property if and only if it is relatively uniquely ergodic and confined.
\end{prop} 
This equivalence relies on the canonical relation between quasifactors (see \cite[Chapter 8]{glasner}) and joinings.

\begin{proof}
The first part of the proof will be similar to \cite[Proposition 2.1]{GW_paper}. Assume that $\bfX \overset{\pi}{\to} \bfY$ has the GW-property. Let $\bfZ := (Z, \C, \rho, R)$ be a dynamical system and $\l$ a measure which gives us a joining $\bfX \times _\l \bfZ$ where $\bfY$ and $\bfZ$ are independent. We decompose $\l$ over $\bfZ$:
$$\l = \int \mu_z \otimes \delta_z \, d \rho(z).$$
We set $\phi: z \mapsto \mu_z$ and $\g := \phi_* \rho \in \P(\P(X))$. Since $\bfY$ and $\bfZ$ are independent, we have $\pi_*\mu_z = \nu$, $\rho$-almost surely. So $\g$ is supported on $\M_\nu$. Moreover, $\phi$ satisfies the equivariance condition: $\mu_{Rz} = T_* \mu_z$. Therefore, $\g$ is $T_*$-invariant. Finally, the GW-property implies $\g = \delta_\mu$, which in turn yields $\mu_z = \mu$, $\rho$-almost surely. This shows that $\l = \mu \otimes \rho$, so the extension is confined.

We now prove that $\bfX \overset{\pi}{\to} \bfY$ is RUE: take a $T$-invariant measure $\l$ such that $\pi_*\l = \nu$. Then $\l \in \M_\nu$ and it is a fixed point of $\T_*$, therefore, $\delta_\l$ is a $T_*$-invariant measure, so, by the GW property, $\delta_\l = \delta_\mu$. Therefore, $\l = \mu$, and this proves that the extension is RUE.

Conversely, assume that $\bfX \overset{\pi}{\to} \bfY$ is confined and RUE. Let $\rho$ be a $T_*$-invariant measure on $\M_\nu$, and define the associated quasifactor: $\bfZ := (\M_\nu, \rho, T_*)$. We can then use the induced joining
$$\l := \int \tilde \mu \otimes \delta_{\tilde\mu} \, d\rho (\tilde\mu) \in \P(X \times \M_\nu).$$
It projects to $\rho$ on $\M_\nu$. Let us focus on the projection on $X$: $\int \tilde\mu \, d\rho(\tilde\mu)$. Since $\rho$ is supported on $\M_\nu$ and $T_*$-invariant, this measure projects to $\nu$ on $Y$, and is $T$-invariant. Now, using our relative unique ergodicity assumption, this means that $\int \tilde\mu \, d\rho(\tilde\mu) = \mu$. In conclusion, $\rho$ is a joining of $\bfX$ and $\bfZ$. Moreover, we have the computation
\begin{align*}
{(\pi \times Id_\bfZ)}_*\l &= \int \pi_*\tilde\mu \otimes \delta_{\tilde\mu} \, d\rho(\tilde\mu)\\
&= \int \nu \otimes \delta_{\tilde\mu} \, d\rho(\tilde\mu) = \nu \otimes \rho.
\end{align*}
Then, since $\bfX \overset{\pi}{\to} \bfY$ is confined, it follows that $\l = \mu \otimes \rho$. Given the construction of $\l$, it means that $\rho$-almost surely, $\tilde \mu = \mu$, which we can write as $\rho = \delta_\mu$.
\end{proof}

\subsubsection{Relatively uniquely ergodic models of confined extensions}

For a given extension $\A \arr \B$ on $\bfX := (X, \A, \mu, T)$, we can find extensions defined on other dynamical systems that are isomorphic to $\A \arr \B$. Moreover, as we have already mentioned, once the system $\bfX$ is chosen, there exists a system $\bfY := (Y, \C, \nu, S)$ and a factor map $\pi: \bfX \to \bfY$ such that $\B = \pi^{-1}(\C)$ mod $\mu$, but $\pi$ and $\bfY$ are not unique. Once $\bfX$, $\bfY$ and $\pi$ are fixed, we say that $\bfX \overset{\pi}{\to} \bfY$ is a \emph{model} of $\A \arr \B$.  When studying properties invariant under isomorphism like confinement or standardness, the exact choice of the model of $\A \arr \B$ has no impact on our results. 

However, stability and the GW-property are not invariant under isomorphism, and are therefore specific to one model. The purpose of this section is to determine which of the differences between confinement and stability or the GW-property are solely due to a choice of the model. We state our result in the following theorem, where we see, in particular, that, up to the choice of the model, confinement and the GW-property are equivalent.
\begin{thm} Let $\bfX \overset{\pi}{\to} \bfY$ be an extension and let $n\geq 1$. We have
\begin{enumerate} [label = (\roman*)]
\item $\bfY$ is weakly mixing and $\bfX \overset{\pi}{\to} \bfY$ is confined if and only if $\bfX \overset{\pi}{\to} \bfY$ has a $n$-fold self-stable model.
\item If $\bfY$ is ergodic and $\bfX \overset{\pi}{\to} \bfY$ is confined, then $\bfX \overset{\pi}{\to} \bfY$ has a stable model.
\item $\bfX \overset{\pi}{\to} \bfY$ is confined if and only if it has a model which has the GW-property.
\end{enumerate}
\end{thm}

The key assumption on the choice of a model that emerges in Propositions \ref{lem:stable-confined} and \ref{prop:GW_confined} is relative unique ergodicity. Therefore, we want to show that

\begin{prop} \label{prop:model_confined_ext}
A confined extension is isomorphic to a relatively uniquely ergodic extension. 
\end{prop}

We recall that an extension $\A \arr \B$ on $\bfX := (X, \A, \mu, T)$ is relatively ergodic if every $\A$-measurable $T$-invariant set is in $\B$. 

Our proof of Proposition \ref{prop:model_confined_ext} will be done in two steps: first we show in Lemma \ref{lem:confined_rel_ergodic} that confined extensions are relatively ergodic, and then we show in Lemma \ref{lem:rel_ergodic} that a relatively ergodic extension admits a relatively uniquely ergodic model.

\begin{lem} \label{lem:confined_rel_ergodic}
A confined extension is relatively ergodic. 
\end{lem}
\begin{proof}
Let $\bfX := (X, \A, \mu, T)$ be a dynamical system, $\A \arr \B$ a confined extension and denote $\I_X$ the invariant factor of $\bfX$. We aim to show that $\I_X \subset \B$. Set $\I_B := \I_X \cap \B$. Since $T$ acts as the identity map on $\I_X$, the extension $\I_X \arr \I_B$ admits a super-innovation (see Proposition \ref{prop:super-innov_static}): there exist a probability space $(\Omega, \E, \bbP)$, a measure preserving map $p: \Omega \to \quotient{\bfX}{\I_X}$ and a $\sigma$-algebra $\C \subset \E$ independent of $p^{-1}(\I_B)$ such that $p^{-1}(\I_X) \subset p^{-1}(\I_B) \vee \C$. We now want to use $\C$ to get a super-innovation for the extension $\B \vee \I_X \arr \B$. 

Viewing $\mathbf{\Omega} := (\Omega, \E, \bbP, \mathrm{Id})$ as a dynamical system, we can set $\bfZ := (Z, \D, \rho, R)$ to be the relative product of $\bfX$ and $\mathbf{\Omega}$ over $\I_X$, and let $\tilde\C$ be the copy of $\C$ on $\bfZ$. We then get, on $\bfZ$:
\begin{eq} \label{eq:super-innov_invariants}
\B \vee \I_X \subset \B \vee \tilde\C.
\end{eq}
Moreover, $\tilde\C$ is independent of $\B$: let $C$ be a $\tilde\C$-measurable random variable, and consider $\bbE[C \, | \, \B]$. Since $\B$ is an invariant $\sigma$-algebra, we have
$$\bbE[C \, | \, \B] \circ R = \bbE[C \circ \mathrm{Id} \, | \, \B] = \bbE[C \, | \, \B].$$
Therefore,  $\bbE[C \, | \, \B]$ is $\I_B$-measurable, which means that $\bbE[C \, | \, \B] = \bbE[C \, | \, \I_B] = \bbE[C]$, so $C$ is independent of $\B$. Combining this with \eqref{eq:super-innov_invariants} shows that the extension $\B \vee \I_X \arr \B$ admits a super-innovation. However, it is also confined, so Proposition \ref{prop:extension_not_confined} implies that it is a trivial extension: $\B \vee \I_X = \B$ mod $\mu$ and that yields $\I_X \subset \B$ mod $\mu$.
\end{proof}

We continue our proof of Proposition \ref{prop:model_confined_ext} with the following lemma:

\begin{lem} \label{lem:rel_ergodic}
Any relatively ergodic extension $\A \arr \B$ has a relatively uniquely ergodic model. 
\end{lem}
The case when $\B$ (and therefore $\A$) is ergodic is already known: it is a result from Weiss (\cite{strict_erg_models}). Weiss's result gives a stronger conclusion than ours since it gives a \emph{topological} RUE model. In the non-ergodic case, we only get a model $\bfX \overset{\pi}{\to} \bfY$ where $X$ and $Y$ are standard Borel spaces and $\pi$ is a Borel map. One could try to improve this result and build a model where $\bfX$ and $\bfY$ are topological systems by making use of Weiss and Downarowicz's result from \cite{pure_models}. By improving the result from \cite{pure_models}, one might also be able to get a model where $\pi$ is continuous.

\begin{proof}
Let $\A \arr \B$ be a relatively ergodic extension on $\bfX := (X, \A, \mu, T)$. Start by taking a system $\bfY := (Y, \C, \nu, S)$ and a factor map $\pi: \bfX \to \bfY$ such that $\B = \s(\pi)$ mod $\mu$. Since $(X, \A)$ and $(Y, \C)$ are standard Borel spaces, we can chose Polish topologies $\scrT_X$ and $\scrT_Y$ such that $\A = \s(\scrT_X)$ and $\C = \s(\scrT_Y)$.

We recall that a point $x$ is generic if there exists an invariant measure $\Phi_X(x)$ such that for every continuous bounded function $f: X \arr \bbR$, we have
\begin{eq} \label{eq:generic}
\frac{1}{n}\sum_{k=0}^{n-1} f \circ T^k(x) \underset{n \arr \infty}{\to} \int f d\Phi_X(x).
\end{eq}
Using Birkhoff's ergodic theorem, we know that on $\bfX$ and $\bfY$, almost every point is generic for an ergodic measure. Therefore, there are invariant measurable subsets $Y_0 \subset Y$ and $X_0 \subset \pi^{-1}Y_0 \subset X$ of full measure on which all points are generic for an ergodic measure. We can define the map $\Phi_X: X_0 \to \P(X)$ that sends each point to the ergodic measure for which it is generic and for all $x \in X_0$, $\Phi_X(x)(X_0) = 1$. Finally, define $\Phi_X$ on $X \backslash X_0$ as any measurable map, which implies that $\Phi_X: X \to \P(X)$ is measurable. We define $\Phi_Y:Y \to \P(Y)$ similarly.

Those maps are measurable and generate the respective invariant factor $\sigma$-algebras of $\bfX$ and $\bfY$. Indeed, let us show that $\sigma(\Phi_X) = \I_X$ mod $\mu$: it is clear that $\sigma(\Phi_X) \subset \I_X$ mod $\mu$, so we need to show the converse. This will follow from the equality, for any bounded measurable function $f$:
\begin{eq}\label{eq:Phi}
\int f d \Phi_X(x) = \lim_{n \arr \infty} \frac{1}{n} \sum_{k=0}^{n-1} f \circ T^k x = \bbE[f \, | \, \I_X](x), \; \mu \text{-almost surely}.
\end{eq}
We get \eqref{eq:Phi} by first showing it for continuous functions, and then extending it to all bounded measurable functions. Once that is established, take $f$ a $T$-invariant function and use \eqref{eq:Phi} to get
$$\int f d \Phi_X(x) = \bbE[f \, | \, \I_X](x) = f(x), \; \mu \text{-almost surely}.$$
This shows that $f$ is $\Phi_X$-measurable, therefore completing the proof that $\sigma(\Phi_X) = \I_X$ mod $\mu$.

Combining that with the relative ergodicity of $\A \arr \B$, we have
$$\sigma(\Phi_X) = \I_X = \pi^{-1}(\I_Y) = \pi^{-1}\sigma(\Phi_Y) = \sigma(\Phi_Y \circ \pi) \text{ mod } \mu.$$
Moreover, the laws of $\Phi_X$ and $\Phi_Y$, which we denote $\xi_X$ and $\xi_Y$ respectively, give the ergodic decompositions of $\mu$ and $\nu$. Therefore, there exists a measurable map $\Gamma: (\P(Y), \xi_Y) \to (\P(X), \xi_X)$ such that 
$$\Phi_X = \Gamma \circ \Phi_Y \circ \pi \; \mu \text{-almost surely}.$$
Moreover, it is easy to verify that $\Phi_Y \circ \pi = \pi_* \circ \Phi_X$ on $X_0$. Finally, it yields that 
$$\Gamma \circ \pi_* = \rm{Id}_{\P(X)} \; \xi_X \text{-almost surely}.$$
From that, we can use \cite[Corollary 15.2]{kechris} to deduce that there are two measurable sets $\Sigma_X \subset \P(X)$ and $\Sigma_Y \subset \P(Y)$ of full measure (for $\xi_X$ and $\xi_Y$ respectively) such that $\pi_*: \Sigma_X \to \Sigma_Y$ is a measurable bijection (with $\pi_*^{-1} = \Gamma$). Finally, set $X_1 := \Phi_X^{-1} (\Sigma_X) \cap X_0$ and $Y_1 := \Phi_Y^{-1}(\Sigma_Y) \cap Y_0$ and define $\bfX_1$ (resp. $\bfY_1$) as the restriction of $\bfX$ to $X_1$ (resp. $\bfY$ to $Y_1$). This is well defined because $X_1$ and $Y_1$ are invariant measurable sets that verify $\pi(X_1) \subset Y_1$. We know that $(X_1, \A_1) := (X_1, \A \cap X_1)$ and $(Y_1, \C_1) := (Y_1, \C \cap Y_1)$ are still standard Borel spaces because they are the restriction of a standard Borel space to a measurable subset. 

Define $\Phi_{X_1}: X_1 \to \P(X_1) \cup \{0\}$ by restricting $\Phi_X(x)$ to $X_1$. This is well defined because, since $\Phi_{X}(x)$ is ergodic and $X_1$ is invariant, the restriction of $\Phi_X(x)$ to $X_1$ is either a probability measure or the null measure.

The extension $\bfX_1 \overset{\pi_1}{\to} \bfY_1$ gives us the desired model. Let $\l$ be a $T$-invariant probability measure on $X_1$ such that $\pi_*\l = \nu_1$ and set $\chi := (\Phi_{X_1})_*\l$. First, consider $\tilde\l(\cdot) := \l(\cdot \cap X_1)$ and integrate \eqref{eq:generic} over $X$: for any bounded $\scrT_X$-continuous $f$ 
$$\int_X f \tilde\l = \int_{X_1} f d\l = \int_{X_1} \int f d\Phi_X(x) \, d\l(x),$$
so
$$\tilde\l = \int_{X_1} \Phi_X(x) d\l(x).$$
So $\Phi_X(x)(X_1) = 1$ $\l$-almost surely and
$$\l = \int_{X_1} \Phi_{X_1}(x) d\l(x) = \int \rho \; d\chi(\rho).$$
By applying $\pi_*$, we get
$$\nu_1 = \int \pi_*\rho \; d\chi(\rho),$$
which yields $(\pi_*)_*\chi = \xi_{Y_1}$, by uniqueness of the ergodic decomposition of $\nu_1$. Then, since $\pi_*$ is a bijection, we get
$$\chi = (\pi_*^{-1})_*\xi_{Y_1} = \xi_{X_1},$$
and finally 
$$\l = \int \rho \; d \chi(\rho) = \int \rho \; d\xi_{X_1} = \mu_1.$$
\end{proof}

\subsection{Compact extensions: a criterion for confinement} \label{sect:compact}

In this section, we prove the confinement criterion presented in Theorem \ref{thm:confinement_compact} and give applications of this criterion.

\subsubsection{A criterion for confinement of compact extensions}

We recall our criterion. For a compact extension $\bfX := \bfY \ltimes_\phi G \to \bfY$, we want to show that the following are equivalent:
\begin{enumerate} [label= (\roman*)]
\item The extension $\bfX \overset{\pi}{\to} \bfY$ is confined;
\item The product extension $\bfX \otimes \bfX \overset{\pi \times \pi}{\to} \bfY \otimes \bfY$ is relatively ergodic;
\item For $\bbP$-almost every $\o$ the measure $\rho_\o \otimes m_G \otimes m_G$ is ergodic under $(S \times S)_{\phi \times \phi}$.
\end{enumerate}
As we see in the proof below, the implications (i) $\Rightarrow$ (ii) and (ii) $\Rightarrow$ (iii) are general results for which the extension does not need to be compact. However, the implication (iii) $\Rightarrow$ (i) uses the compactness of the extension, because the main ingredient in our proof is Furstenberg's well known relative unique ergodicity result:
\begin{lem} [Furstenberg \cite{Furstenberg_disjointness}] \label{lem:rel_unique_ergo}
Let $\bfX := (Y \times G, \nu \otimes m_G, S_\phi)$ be the system introduced in Definition \ref{def:compact_ext} and $\pi: (y, g) \mapsto y$. If $\bfX$ is ergodic, then, for any $S_\phi$-invariant measure $\l$ on $Y \times G$ which verifies $\pi_*\l = \nu$, we have $\l = \nu \otimes m_G$. 
\end{lem}
See \cite[Theorem 3.30]{glasner} for a proof of the lemma.

\begin{proof}[Proof of Theorem \ref{thm:confinement_compact}]
\begin{description}[leftmargin=*]
\item[(i) $\Rightarrow$ (ii).] Assume that $\bfX \overset{\pi}{\to} \bfY$ is a confined compact extension. From the definition of confined extensions, we know that the product extension $\bfX \otimes \bfX \overset{\pi \times \pi}{\to} \bfY \otimes \bfY$ is also confined. So it is relatively ergodic (see Lemma \ref{lem:confined_rel_ergodic}).
\item[(ii) $\Rightarrow$ (iii).] It follows from the lemma, which is true for any extension given by a Rokhlin cocycle:
\begin{lem} \label{lem:compact_RUE}
Let $\bfY := (Y, \B, \nu, S)$ be a dynamical system, $(Z, \rho)$ a standard Borel space and $R_\bullet: Y \to \mathrm{Aut}(Z, \rho)$ a measurable cocycle. Consider the Rokhlin extension $\bfX$ defined on $(X, \mu) := (Y \times Z, \nu \otimes \rho)$ by the transformation
$$T: (y, z) \mapsto (Sy, R_y(z)).$$
Consider
$$\nu = \int \t\nu \, d\chi_\nu(\t\nu),$$
the ergodic decomposition of $\nu$. If $\bfX \overset{\pi}{\to} \bfY$ is relatively ergodic, then the ergodic decomposition of $\nu \otimes \rho$ is
$$\nu \otimes \rho = \int \t\nu \otimes \rho \, d\chi_\nu(\t\nu).$$
\end{lem}
\begin{proof}
Let $\I_T$ be the $\s$-algebra of $T$-invariant sets on $Y \times Z$ and $\I_S$ the $\s$-algebra of $S$-invariant sets on $Y$. By assumption, $\pi^{-1}\I_S = \I_T$, so if we take $f: Y \to \bbR$ and $h: Z \to\bbR$ bounded measurable functions, the independence of $Y$ and $Z$ gives
$$\bbE_{\nu \otimes \rho} [f(y) h(z) \, | \, \I_T] = \bbE_\nu[f(y) \, | \, \I_S] \cdot \bbE_\rho[h(z)].$$
Since the ergodic decomposition is obtained by decomposing over the factor of invariant sets, this equality implies that
$$\int \t \nu \otimes \rho \, d\chi_\nu(\t\nu),$$
gives the ergodic decomposition of $\nu \otimes \rho$.

\end{proof}
\item[(iii) $\Rightarrow$ (i).] Assume that the condition (iii) holds. Let $\l$ be a self-joining of $\bfX$ for which the copies of $\bfY$ are independent, i.e. $\l$ projects to $\nu \otimes \nu$ on $Y \times Y$. We consider the ergodic decomposition of $\l$ 
$$\l = \int \l_\o \, d \, \bbP(\o).$$
Then, since $\l$ projects onto $\nu \otimes \nu$, we know that $\rho_\o := (\pi\times \pi)_*\l_\o$ gives an ergodic decomposition of $\nu \otimes \nu$. Therefore, by the uniqueness of the ergodic decomposition and our hypothesis, we get that, for $\bbP$-almost every $\o$, the measure $\rho_\o \otimes m_{G \times G} = \rho_ \o \otimes m_G \otimes m_G$ is ergodic. Moreover, we can see $(X \times X, \rho_\o \otimes m_{G \times G}, T \times T)$ as a compact extension of $(Y \times Y, \rho_\o, S \times S)$ via the co-cycle $\psi: (y, y') \mapsto (\phi(y), \phi(y'))$ taking values in the compact group $G \times G$. Since $\rho_\o \otimes m_{G \times G}$ is ergodic, Lemma \ref{lem:rel_unique_ergo} tells us that $\l_\o = \rho_\o \otimes m_{G \times G}$. Finally:
$$\l = \int \rho_\o \, d \, \bbP(\o) \otimes m_{G \times G} = \nu \otimes \nu \otimes m_{G} \otimes m_G = \mu \otimes \mu.$$

\end{description}
\end{proof}

\subsubsection{A non-confined ergodic compact extension}
 \label{exe:non-confined_compact_ext}
We give here a simple example of non-confined ergodic compact extension: take $\bfY := (Y, \B, \nu, S)$ a weakly mixing dynamical system, $\a$ an irrationnal number and the constant cocycle 
$$\begin{array}{lccc}
\phi: & Y  &\to & \bbT\\
& y & \mapsto & \a 
\end{array}$$
Then the associated compact extension $(Y \times \bbT, \nu \otimes m_\bbT, S_\phi)$ defined by 
$$S_\phi : (y, z) \mapsto (Sy, z + \a)$$
is ergodic. Also, because the cocycle is constant, the extension is of product type, and therefore it is not confined.

\subsubsection{A non-weakly mixing confined compact extension}
Let us now turn our attention to an example illustrating our confinement criterion in the non-weakly mixing case. We consider the system $\bfX_\a$ on the two dimensional torus given by the following Anzai product:
$$T_\a : (x, y) \mapsto (x + \alpha, y + x),$$
as an extension of the system $\bfY_\a$ given by the rotation 
$$S_\a: x \mapsto x + \alpha,$$
with $\alpha \in \bbR$. We equip both systems with the Lebesgue measure.
\begin{prop} \label{prop:non-WM_compact_ext}
The system $\bfX_\a$ is not weakly mixing but the extension $\bfX_\a \overset{\pi}{\to} \bfY_\a$ is compact and confined.
\end{prop}

\begin{proof}[Proof of Proposition \ref{prop:non-WM_compact_ext} using Theorem \ref{thm:confinement_compact}]
We know that $\bfX_\a$ is not weakly mixing because $\bfY_\a$ is not. The extension $\bfX_\a \overset{\pi}{\to} \bfY_\a$ is clearly compact. Let us show that it is confined.

Denote by $\nu$ the Lebesgue measure on the torus $\bbT$ and by $\mu := \nu \otimes \nu$ the Lebesgue measure on the $2$-dimensional torus, $\bbT^2$. For $\o \in \bbT$, define $\mu_\o := \int \delta_x \otimes \delta_{x +\o} d\nu(x)$. The measures $(\mu_\o)_{\o \in \bbT}$ give an ergodic decomposition of $\nu\otimes \nu$:
$$\nu \otimes \nu = \int \mu_\o d\nu(\o).$$
In light of Theorem \ref{thm:confinement_compact}, we need to show that, for $\nu$-almost every $\o$, the system $(\bbT^4, \mu_\o \otimes \nu \otimes \nu, T_\a \times T_\a)$ is ergodic. This is isomorphic to the system on $(\bbT^3, \nu \otimes \nu \otimes \nu)$ given by 
\begin{eq} \label{eq:ergodic_comp}
(x, y_1, y_2) \mapsto (x + \a, y_1 + x, y_2 + x + \o),
\end{eq}
which is a compact extension of $\bfY_\a$ via the cocycle 
$$\phi: x \mapsto (x, x + \o).$$
It is known that this yields an ergodic system if an only if there exist $(n_1, n_2) \in \bbZ\backslash \{(0, 0)\}$ and a measurable map $f: \bbT \to \bbT$ such that 
$$e^{2i \pi n_1x} e^{2i \pi n_2 (x + \o)} = f(x + \a) / f(x) \; \text{ for $\nu$-almost all $x$}.$$
By considering the Fourier series of such a function $f$, we see that this is only possible if there is $k \in \bbZ$ such that $n_2 \o - k \a \in \bbZ$.

Therefore, this means that there is a countable number of $\o \in \bbT$ such that \eqref{eq:ergodic_comp} yields a non-ergodic system. Since $\nu$ is non-atomic, we conclude that for $\nu$-almost every $\o$, the system $(\bbT^4, \mu_\o \otimes \nu \otimes \nu, T_\a \times T_\a)$ is ergodic, and, from Theorem \ref{thm:confinement_compact}, we know that $\bfX_\a \overset{\pi}{\to} \bfY_\a$ is confined.
\end{proof}

Let us add the following result regarding this extension:
\begin{defi} \label{def:rigid}
A dynamical system $\bfX := (X, \A, \mu, T)$ is rigid if there exists a sequence $(n_k)_{k \geq 0}$ that goes to $\infty$ such that, for every measurable set $A \subset X$, we have
$$\lim_{k \rightarrow \infty} \mu(T^{n_k}A \, \Delta \, A) =0.$$
For such a sequence, we say that $\bfX$ is $(n_k)$-rigid.
\end{defi}
\begin{prop} \label{prop:non_rigid_confined_extension}
The factor $\bfY_\a$ is rigid but the extension $\bfX_\a$ is not.
\end{prop}
The fact that $\bfX_\a$ is not rigid follows easily from \cite[Theorem 6]{Anzai}.

\section{$T, T^{-1}$ transformations} \label{sect:TT-1}

Let $\bfY := (Y, \B, \nu, S)$ be a Bernoulli shift with $Y = \{-1, 1\}^\bbZ$, $\nu = \frac{1}{2}(\delta_{-1} + \delta_{1})^{\otimes \bbZ}$ and $S$ the shift on $Y$. 
Let $\bfX := (X, \A, \mu, T)$ be a dynamical system. We introduce the system $\bfY \ltimes \bfX$ defined on $(Y \times X, \nu \otimes \mu)$ by the transformation 
$$\begin{array}{cccc}
S \ltimes T: & Y \times X & \to & Y \times X\\
& (y, x) & \mapsto & (Sy, T^{y(0)}x)
\end{array}.$$
We call this map a $T, T^{-1}$ transformation.

Most of our approach in Section \ref{sect:T, T-1_cocycle} is based on the arguments given in \cite{erg_rokhlin_cocycles}, with the necessary adjustments to apply them to $T, T^{-1}$ transformations and confined extensions. Then in Section \ref{sect:non-standard}, we use new arguments to finally get a non-standard extension.

\subsection{Confinement result for $T, T^{-1}$ transformations} \label{sect:T, T-1_cocycle}

In this section, we will prove that
\begin{thm} \label{thm:TT-1_confined}
If $T^2$ acts ergodically on $(X, \A, \mu)$, then $\pi: \bfY \ltimes \bfX \to \bfY$ is a confined extension. 
\end{thm}
As a consequence, we will get 
\begin{cor} \label{cor:TT-1_relWM}
If $T$ is weakly mixing, then $\pi: \bfY \ltimes \bfX \to \bfY$ is confined but not compact.
\end{cor}
Since all the confined extensions we built previously were compact, this gives us new examples.

\subsubsection{Ergodic properties of the cocycle}

We will prove the theorem using the framework proposed in \cite[Section 6]{erg_rokhlin_cocycles}. To do so, we first remark that the transformation $(S \ltimes T) \times (S \ltimes T)$ can easily be described using a cocycle. In order to see that, we define the $\bbZ^2$-action on $X \times X$ by
$$T_{(k_1, k_2)}(x_1, x_2) := (T^{k_1}x_1, T^{k_2}x_2),$$
and then $(S \ltimes T) \times (S \ltimes T)$ is isomorphic to
$$\begin{array}{cccc}
S_{2, \phi}: & Y^2 \times X^2 & \to & Y^2 \times X^2\\
& (y, x) & \mapsto & (S_2(y), T_{\phi(y)}(x))
\end{array}$$
where $S_2 := S \times S$ and $\phi$ is the following cocycle
$$\begin{array}{cccc}
\phi: & Y^2 & \to & \bbZ^2\\
 & (y_1, y_2) & \mapsto & (y_1(0), y_2(0))\\
\end{array}.$$
We then have to study the ergodic properties of the associated transformation 
$$\begin{array}{cccc}
\bar S_{2, \phi}: & Y^2 \times \bbZ^2 & \to & Y^2 \times \bbZ^2\\
 & (y, k) & \longmapsto & \left(S_2(y), k + \phi(y)\right)\\
\end{array}.$$
 We will also note $\nu_2 := \nu \otimes \nu$.

On $\bbZ^2$, we introduce $Z_0 := \{k \in \bbZ^2 \, | \, k_1 -k_2 \text{ is even}\}$ and $Z_1 := \{k \in \bbZ^2 \, | \, k_1 -k_2 \text{ is odd}\}$. We point out that, for every $y \in Y^2$, we have
\begin{eq}
Z_0 + \phi(y) = Z_0 \; \text{ and } \; Z_1 + \phi(y) = Z_1.
\end{eq}
Finally, for any subset $F \subset \bbZ^2$, we denote by $\l_F$ the counting measure on $F$. We then have the following result, which can be viewed as a two dimensional version of the result in \cite[\S 12]{SchmidtCocycles}:

\begin{lem} \label{lem:ergodicite_cocycle}
The ergodic components of $(Y^2 \times \bbZ^2, \bar S_{2, \phi}, \nu_2 \otimes \l_{\bbZ^2})$ are $\nu_2 \otimes \l_{Z_0}$ and $\nu_2 \otimes \l_{Z_1}$.
\end{lem}

\begin{proof} 
We need to show that $\nu_2 \otimes \l_{Z_0}$ and $\nu_2 \otimes \l_{Z_1}$ are ergodic. Let us consider the case of $\nu_2 \otimes \l_{Z_0}$.

We could adapt the arguments in \cite[\S 12]{SchmidtCocycles} to a $2$-dimensional setting. We can also use the following map:
$$\begin{array}{cccc}
\psi: & Y^2 \times Z_0 & \to & {Z_0}^\bbZ\\
& (y, k) & \longmapsto & (k + \phi_n(y))_{n \in \bbZ}
\end{array},$$
with 
$$\phi_n(y) := \left\{
\begin{array}{cl}
0 & \text{ if } n=0\\
\sum_{j = 0}^{n-1} \phi(S_2^j(y)) & \text{ if } n > 0\\
\sum_{j = n}^{-1} \phi(S_2^j(y)) & \text{ if } n < 0\\
\end{array} \right..
$$
The map $\psi$ sends $\bar S_{2, \phi}$ onto the shift on $Z_0^\bbZ$. Moreover, $\psi_*(\nu_2 \otimes \l_{Z_0})$ is the shift-invariant measure on $Z_0^\bbZ$ obtained by applying the symmetric random walk on the counting measure $\l_{Z_0}$. Therefore our system is isomorphic to the infinite measure preserving system associated to the symmetric random walk on $Z_0$, which is known to be a conservative and ergodic system (see \cite[Theorem 4.5.3]{aaronson}). 
\end{proof}
\begin{cor} \label{cor:ergodicite_cocycle}
Define the map
$$\begin{array}{cccc}
H: & Y^2 & \to & Y^2\\
& y & \longmapsto & S_2^{N_1(y)}(y)\\
\end{array},$$
where $N_1$ is the first return time to $(0, 0)$ of $(\phi_n)_{n \geq 0}$. Then $H$ is well defined, measure preserving and ergodic on $(Y^2, \nu_2)$.
\end{cor}
\begin{proof}
The map $H$ is simply isomorphic to the map induced by the system $(Y^2 \times Z_0, \nu_2 \otimes \l_{Z_0}, \bar S_{2, \phi})$ on $Y^2 \times \{(0, 0)\}$. Since we have seen that this system is conservative and ergodic, we get our corollary.
\end{proof}
Actually, a closer study of $H$ would show that it is isomorphic to a Bernoulli shift, but this ergodicity result will be enough for our purposes.

\subsubsection{Confinement for $T, T^{-1}$ transformations} \label{sect:confined_T,T-1}

Using our previous section and the arguments taken from \cite{erg_rokhlin_cocycles}, we now prove Theorem \ref{thm:TT-1_confined}.

Lema\'nczyk and Lesigne gave a condition (see \cite[Propostion 8]{erg_rokhlin_cocycles}) for ergodic cocycles to yield stable extensions. Our proof here is a straightforward adaptation that takes into account the lack of ergodicity of the cocycle and shows that the extension is confined. We give a detailed proof for the sake of completeness.
\begin{proof}[Proof of Theorem \ref{thm:TT-1_confined}]
Let $\l$ be a $(S \ltimes T) \times (S \ltimes T)$-invariant self-joining of $\bfY \ltimes \bfX$ for which the projection on $Y \times Y$ is the product measure $\nu_2 = \nu \otimes \nu$. As we remarked previously $\l$ being $(S \ltimes T) \times (S \ltimes T)$-invariant is, up to a permutation of the coordinates, equivalent to $\l$ being $S_{2, \phi}$-invariant.

Let us decompose $\l$ over the projection on $Y^2$: 
$$\l = \int_{Y^2} \delta_y \otimes \mu_y \, d\nu_2(y).$$
Since $\l$ is $S_{2, \phi}$-invariant, we get $\mu_{S_2(y)} = (T_{\phi(y)})_*\mu_y$. Then we set
$$F: (y, k) \mapsto (T_k^{-1})_*\mu_y$$
which is $\bar S_{2, \phi}$-invariant. So, using Lemma \ref{lem:ergodicite_cocycle}, we know that it is almost surely constant on $Y^2 \times Z_0$ and $Y^2 \times Z_1$. In particular, for $\nu_2$-almost every $y$, we have $\mu_y = F(y, 0) = \gamma_0$, where $\gamma_0$ is a probability measure on $X^2$. Therefore $\l = \nu_2 \otimes \gamma_0$. 
Moreover, since, for almost every $y$, $F(y, \cdot)$ is constant on $Z_0$, by taking $k = (2, 0)$ and $k = (0, 0)$, we get
$$((T^{2} \times \rm{Id})^{-1})_*\gamma_0 = \gamma_0,$$
which we can also write as
$$(T^{2} \times \rm{Id})_*\gamma_0 = \gamma_0.$$
Since both marginals of $\gamma_0$ are $\mu$, it implies that $\gamma_0$ is a joining of $(X, \mu, T^2)$ and $(X, \mu, \rm{Id})$. However, all ergodic transformations are disjoint from the identity map, so $\gamma_0$ is a product measure, and more precisely $\gamma_0 = \mu \otimes \mu$. This means that 
$$\l = \nu \otimes \nu \otimes \mu \otimes \mu.$$

\end{proof}
\begin{rmq}
The ergodicity assumption on $T^2$ is necessary to get the result of the theorem. Indeed, for example, if $T^2 = \rm{Id}$ we get $T = T^{-1}$, and then $S \ltimes T = S \times T$, which is clearly not confined (see Corollary \ref{cor:extension_not_confined}). In fact, for any transformation for which $T^2$ is not ergodic, the $T, T^{-1}$ extension is not confined. Indeed, take $f$ a non-trivial $T^2$-invariant function on $X$. Set $\xi := (f, f \circ T)$ and $\tilde\xi(y, x) := \xi(x)$. Since $f \circ T^2 = f$, we also get that $f \circ T^{-1} = f \circ T$. Using this, we check that $\sigma(\tilde \xi)$ is invariant under the $T, T^{-1}$ transformation. But, by construction, it is independent of the $Y$ coordinate, which means (using again Corollary \ref{cor:extension_not_confined}) that the extension is not confined.


Our additional ergodicity assumption on $T^2$ is here to compensate the lack of ergodicity from the cocycle that arises from the fact that the random walk set by $\phi$ is not ergodic. Indeed, if we modify $Y$ and take random variables uniformly distributed on $\{-1, 0, 1\}$, the random walk it generates on $\bbZ^2$ is ergodic, the cocycle is as well and we get the result of the theorem by only assuming that $T$ is ergodic. This is the setup to which Lema\'nczyk and Lesigne applied their stability criterion. 
\end{rmq}

\subsubsection{Relative weak mixing of $T, T^{-1}$ transformations}

In this section, we prove Corollary \ref{cor:TT-1_relWM} using the notion of relative weak mixing (see \cite{Furstenberg_book}).
\begin{defi}
Let $\bfU$ and $\bfV$ be dynamical systems and $\alpha: \bfU \to \bfV$ a factor map. The extension $\bfU \overset{\a}{\to} \bfV$ is relatively weakly mixing if the extension 
$$\bfU \otimes_\bfV\bfU \to \bfV$$ 
is relatively ergodic.
\end{defi}
It is a well known fact that relatively weakly mixing extensions are not compact (see \cite[Chapter 6, Part 4]{Furstenberg_book}). It can also be seen from the construction we give in the proof of Lemma \ref{lem:compact_standard}.

\begin{proof}[Proof of Corollary \ref{cor:TT-1_relWM}]
Now assume that $T$ is weakly mixing. Therefore $(T \times T)^2$ acts ergodically on $(X \times X, \A \otimes \A, \mu \otimes \mu)$. One can check that the system, that we note $\bfY \ltimes (\bfX \otimes \bfX)$, given on $(Y \times X \times X, \nu \otimes \mu \otimes \mu)$ by the transformation 
$$\begin{array}{cccc}
S \ltimes (T \times T): & Y \times X \times X & \to & Y \times X \times X \\
& (y, x, x') & \mapsto & (Sy, T^{y(0)}x, T^{y(0)}x')
\end{array},$$
is the relatively independent product of $\bfY \ltimes \bfX$ over $\bfY$. Using Theorem \ref{thm:TT-1_confined}, we know that $\bfY \ltimes (\bfX \otimes \bfX) \to \bfY$ is confined, and, by Lemma \ref{lem:confined_rel_ergodic}, it is relatively ergodic. It follows that $\bfY \ltimes \bfX \to \bfY$ is relatively weakly mixing, and therefore not compact.
\end{proof}

\subsection{A non-standard $T, T^{-1}$ extension} \label{sect:non-standard}

The goal of this section is to show the following result:

\begin{thm} \label{thm:hyper-confined_extension}
If $\bfX$ has the $4$-fold PID property and $T^2$ acts ergodically on $(X, \A, \mu)$, then the extension $\bfY \ltimes \bfX \overset{\pi}{\to} \bfY$ is not standard.
\end{thm}

We introduced the PID property in Definition \ref{defi:PID}. It is known that there are systems satisfying this property: it is shown in \cite{ryzhikov} that all finite rank mixing transformations have the PID property. We will make use of Ryzhikov's result from \cite[\S 1, Section 2]{ryzhikov}. He proves that if $\bfX$ has the $4$-fold PID property, any joining $\l$ of $\bfX$ with any two systems $\bfZ_1$ and $\bfZ_2$ that is pairwise independent has to be the product joining. In fact, we will only need a simplified version, which we write here as a consequence of Lemma \ref{lem:rel_prod}:
\begin{lem} \label{lem:butterfly}
Assume that $\bfX$ has the $4$-fold PID property. Let $\bfZ$ be a dynamical system and consider a joining $\bfX_1 \times \bfX_2 \times \bfZ$ where $\bfX_1$ and $\bfX_2$ are copies of $\bfX$. If this joining is pairwise independent, then it is the product joining.
\end{lem}
\begin{proof}
Let $\l$ be a pairwise independent joining on $X \times X \times Z$. Take $\bfW$ the system given by the relatively independent product of $\l$ over $\bfZ$, and denote $\bfX_1'$, $\bfX_2'$, $\bfX_1''$ and $\bfX_2''$ the copies of $\bfX$ on $\bfW$. Using our assumption on $\l$ and Lemma \ref{lem:rel_prod}, we know that the quadruplet $(\bfX_1', \bfX_2',\bfX_1'', \bfX_2'')$ is pairwise independent. Next, the $4$-fold PID property tells us that this quadruplet is mutually independent. Therefore, $\bfX_1' \vee \bfX_2'$ and $\bfX_1'' \vee \bfX_2''$ are independent. Finally, using Lemma \ref{lem:rel_prod} once more, we know that $\bfX_1 \vee \bfX_2$ and $\bfZ$ are independent, which implies that $\l$ is the product joining.
\end{proof}

\begin{proof}[Proof of Theorem \ref{thm:hyper-confined_extension}.]
Let us take an extension $\tilde\bfY \overset{\a}{\to} \bfY$. Consider the system $\bfW := \tilde\bfY \ltimes \bfX = (\tilde Y \times X, \tilde \nu \otimes \mu, Q)$ with $Q$ being the map
$$Q: (\tilde y, x) \mapsto (\tilde S \tilde y, T^{y(0)}x),$$
using the notation $y := \alpha(\tilde y)$. One can check that $\bfW$ is the relatively independent product of $\bfY \ltimes \bfX$ and $\tilde\bfY$ over $\bfY$.
Following Remark \ref{rmq*:super-innov}, we need to show that the extension $\tilde\bfY \ltimes \bfX \overset{\tilde\pi}{\to} \tilde\bfY$ admits no super-innovation. 

Assume that $\bfW \overset{\tilde\pi}{\to} \tilde\bfY$ has a super-innovation. Therefore there is a system $\bfZ$ and a joining $\bfW \times \bfZ$ in which $\bfZ$ is independent of $\tilde\bfY$ and $\bfW$ is $\tilde\bfY \vee \bfZ$-measurable. We consider the $2$-fold relatively independent product of $\bfW \times \bfZ$ over $\bfZ$ and denote $\rho$ the measure of the self-joining of $\bfW$ we obtain. Our goal is to show that $\rho$ is the product joining. We represent the construction of $\rho$ in the following diagram:

\vspace{5mm}
\centerline{\includegraphics[scale=1.1]{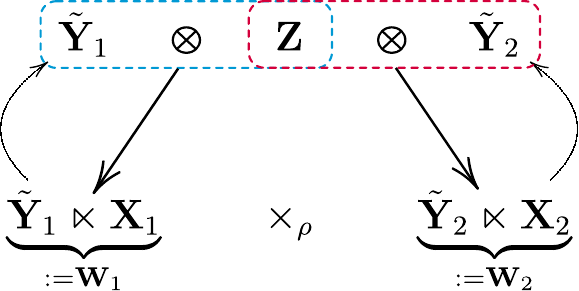}}
\vspace{2mm}
In this joining, we have the two following properties:
\begin{enumerate} [label = (\roman*)]
\item the copies of $\tilde\bfY$ are independent,
\item we have the additional independence: for $i = 1, 2$ we get that $\bfX_i$ is independent of $\tilde\bfY_1 \vee \tilde\bfY_2$.
\end{enumerate}
Indeed, (i) follows from the independence of $\bfZ$ and $\tilde\bfY$. To get (ii) (with $i=1$ for example), first note that $\bfX_1 \vee \tilde\bfY_1$ is $\bfZ \vee \tilde\bfY_1$-measurable, therefore is independent of $\tilde\bfY_2$. Since $\bfX_1$ and $\tilde\bfY_1$ are independent, this yields that $\bfX_1$ is independent of $\tilde\bfY_1 \vee \tilde\bfY_2$, which gives us (ii).

For the rest of this proof, we switch the coordinates and view $\rho$ as a measure on $\tilde Y \times \tilde Y \times X \times X$. We note a point on this space as $(\tilde y_1, \tilde y_2, x_1, x_2)$ or $(\tilde y, x)$, for short. Moreover, we use the notation $y := \a \times \a(\tilde y)$, and this enables us to define $\tilde \phi$, $\tilde \phi_n$ and $\tilde N_1$ as maps on $\tilde Y \times \tilde Y$ by setting
$$\tilde \phi(\tilde y) := \phi(y), \; \tilde\phi_n(\tilde y) = \phi_n(y) \; \text{ and } \; \tilde N_1(\tilde y) := N_1(y).$$
We also define $\hat N_1$ on $\tilde Y \times \tilde Y \times X \times X$ by setting $\hat N_1(\tilde y, x) := N_1(y)$. We now get back to our proof.

From (ii), we know that each $\bfX_i$ is independent of $\tilde\bfY_1 \vee \tilde\bfY_2$. Moreover, from Theorem \ref{thm:TT-1_confined}, we know that $\bfX_1$ and $\bfX_2$ are independent, so, under $\rho$, $\tilde\bfY_1 \vee \tilde\bfY_2$, $\bfX_1$ and $\bfX_2$ are pairwise independent. We then want to use Lemma \ref{lem:butterfly} to get the mutual independence, but we do not have a transformation $\theta$ on $(\tilde Y \times \tilde Y, \tilde\nu_2)$ such that $\rho$ is $(\theta \times T \times T)$-invariant. Therefore, our strategy bellow is, instead of considering $\tilde\bfY_1 \vee \tilde\bfY_2$, to only consider the conditional law of $\bfX_1 \vee \bfX_2$ given $\tilde\bfY_1 \vee \tilde\bfY_2$, which takes its values in $\P(X \times X)$. Our goal is then to find a suitable transformation on $\P(X \times X)$ and an invariant joining on $\P(X \times X) \times X \times X$ in order to finally apply Lemma \ref{lem:butterfly}.

Since $\tilde\bfY_1$ and $\tilde\bfY_2$ are independent, $\rho$ projects onto $\tilde\nu_2 := \tilde\nu \otimes \tilde\nu$ on $\tilde Y^2$. We decompose $\rho$ over $\tilde Y^2$:
$$\rho = \int_{\tilde Y^2} \delta_{\tilde y} \otimes \mu_{\tilde y} \, d\tilde\nu_2(\tilde y).$$
where each $\mu_{\tilde y}$ is a probability measure on $X \times X $. We consider the following action on $\P(X \times X)$ :
$$\forall k \in \bbZ^2, \; \theta_k: \gamma \mapsto (T_{k})_*\gamma.$$
The fact that $\rho$ is invariant yields: 
\begin{eq} \label{eq:F_invariance}
\mu_{\tilde S_2(\tilde y)} = (T_{\tilde\phi(\tilde y)})_*\mu_{\tilde y} = \theta_{\tilde\phi(\tilde y)}\mu_{\tilde y} \; \text{ almost surely},
\end{eq}
where $\tilde\phi(\tilde y) = \phi(y) = (y_1(0), y_2(0)).$

We now set $\mu_\bullet := \mu_{p_{\tilde\bfY_1 \vee \tilde\bfY_2}(\cdot)}$ and consider the triplet $(\mu_\bullet, p_{\bfX_1}, p_{\bfX_2})$. It is invariant under
$$(Q \times Q)^{\hat N_1}.$$
Indeed:
\begin{align*}
(p_{\bfX_1}, p_{\bfX_2}) \circ (Q \times Q)^{\hat N_1(\tilde y, x)}(\tilde y, x) = T_{\phi_{N_1(y)}(y)}(x) = x,
\end{align*}
and, using \eqref{eq:F_invariance}:
$$\mu_\bullet \circ (Q \times Q)^{\hat N_1(\tilde y, x)}(\tilde y, x) = \theta_{\phi_{N_1(y)}(y)} \mu_{\tilde y} = \mu_{\tilde y},$$
because $\phi_{N_1(y)}(y) = 0$.
However, it follows from Corollary \ref{cor:ergodicite_cocycle} that $\bfY_1 \vee \bfY_2$ is ergodic under $(Q \times Q)^{\hat N_1}$. Therefore, using the fact that ergodic transformations are disjoint from any identity map, it implies that $(\mu_\bullet, p_{\bfX_1}, p_{\bfX_2})$ is independent of $\bfY_1 \vee \bfY_2$. Set $\hat\rho := {(\mu_\bullet, p_{\bfX_1}, p_{\bfX_2})}_*\rho$, which is a probability measure on $\P(X \times X) \times X \times X$. Also consider $\hat \rho_y$ the conditional law of $(\mu_\bullet, p_{\bfX_1}, p_{\bfX_2})$ given $\bfY_1 \vee \bfY_2$ under $\rho$. Using \eqref{eq:F_invariance} and the $Q \times Q$-invariance of $\rho$, we get that 
$$\hat\rho_{S^2y} = (\theta_{\phi(y)} \times T_{\phi(y)})_*\hat\rho_y  \; \text{ almost surely}.$$
Moreover, because $(\mu_\bullet, p_{\bfX_1}, p_{\bfX_2})$ is independent of $\bfY_1 \vee \bfY_2$, we know that, $\nu_2$-almost surely, $\hat \rho_y = \hat\rho$. This yields:
$$\hat\rho = (\theta_{\phi(y)} \times T_{\phi(y)})_* \hat\rho \; \text{ almost surely}.$$
In particular, since $\nu_2(\{\phi(y) = (1, 1)\}) > 0$, we get that $\hat \rho$ is $(\theta_{(1, 1)} \times T \times T)$-invariant.


Let us study $\hat\rho$ more closely: using Theorem \ref{thm:TT-1_confined}, we know that, under $\hat\rho$, $p_{\bfX_1}$ and $p_{\bfX_2}$ are independent. Moreover, using the property (ii) introduced earlier, we get that, for $i = 1, 2$, $p_{\bfX_i}$ is independent of $\mu_\bullet$. So $\hat\rho$ is a pairwise independent joining. We then use  Ryzhikov's result from \cite{ryzhikov} as expressed in Lemma \ref{lem:butterfly} and the $4$-fold PID property of $\bfX$ to conclude that $\hat\rho$ is the product joining.

Therefore, $\bfX_1 \vee \bfX_2$ is independent of $\mu_\bullet$, which is only possible if $\bfX_1 \vee \bfX_2$ is independent of $\tilde\bfY_1 \vee \tilde\bfY_2$.

In conclusion:
$$\rho = \tilde\nu \otimes \mu \otimes \tilde\nu \otimes \mu.$$
Using Lemma \ref{lem:rel_prod}, it implies that $\bfZ$ and $\bfW$ are independent and Lemma \ref{lem:independance_mesurabilite} yields that $\bfW$ is $\tilde\bfY$-measurable, which is absurd. So $\bfW \overset{\tilde\pi}{\to} \tilde\bfY$ admits no super-innovation.
\end{proof}

Given our work in Section \ref{sect:confined_no_super-innov}, it would be natural to try to prove that $\tilde \bfY \ltimes \bfX \overset{\tilde\pi}{\to} \tilde\bfY$ has no super-innovation by showing that it is confined. This would show that $\bfY \ltimes \bfX \overset{\pi}{\to} \bfY$ has a stronger property: it would be \emph{hyper-confined}, as we define below
\begin{defi}
Let $\bfX := (X, \A, \mu, T)$ be a dynamical system. We say that $\A \arr \B$ on $\bfX$ is \emph{hyper-confined} if for every $\beta: \tilde\bfX \to \bfX$ and for every extension $\tilde\B \arr \B$ on $\tilde\bfX$ such that $\A \indep_\B \tilde\B$, we have that $\A \vee \tilde\B \arr \tilde\B$ is confined.

Equivalently, an extension given by a factor map $\pi: \bfX \to \bfY$ is \emph{hyper-confined} if, for every extension $\tilde \bfY \overset{\a}{\to} \bfY$, the extension $\bfX \otimes_\bfY \tilde\bfY \overset{\tilde\pi}{\to} \tilde\bfY$ is confined.
\end{defi}
However, in trying to prove that the $T, T^{-1}$ extension is hyper-confined, we get a similar setup to the proof we gave above, but with a self-joining of $\bfW$ that does not need to verify property (ii). Since we did not manage to complete the proof in that more general case, the following question remains open:
\begin{quest}
Is $\pi: \bfY \ltimes \bfX \to \bfY$ hyper-confined ? More generally, is it possible to build a hyper-confined extension ?
\end{quest}

\section{Application to non-standard dynamical filtrations} \label{sect:non-standard_I-cosy}

Let $\bfX := (X, \A, \mu, T)$. A \emph{dynamical filtration} is a pair $(\F, T)$ such that $\F := (\F_n)_{n \leq 0}$ is a filtration in discrete negative time on $\A$ and each $\F_n$ is $T$-invariant. The theory of dynamical filtrations was initiated by Paul Lanthier in (\cite{these_PL}, \cite{article_PL}). The definitions we give in Section \ref{sect:prod_type} for extensions are based on the theory of filtrations we present here, therefore the process is very similar.

\begin{defi}
Let $(\F, T_1)$ be a dynamical filtration on $\bfX_1 := (X_1, \A_1, \mu_1, T_1)$ and $(\G, T_2)$ a dynamical filtration on $\bfX_2 := (X_2, \A_2, \mu_2, T_2)$. We say that $(\F, T_1)$ and $(\G, T_2)$ are isomorphic if there is an isomorphism $\Phi: \quotient{\bfX_1}{\F_0} \arr \quotient{\bfX_2}{\G_0}$ such that, for all $n \leq 0$, $\Phi(\F_n) = \G_n$ mod $\mu_2$.

If $\F$ and $\G$ are defined on the same system $(X, \A, \mu, T)$, we say that $(\F, T)$ is immersed in $(\G, T)$ if for every $n \leq 0$, $\F_n \subset \G_n$ and we have the following relative independence:
$$\F_{n+1} \indep_{\F_n} \G_n.$$
In general, we say that $(\F, T_1)$ is immersible in $(\G, T_2)$ if there is a dynamical filtration isomorphic to $(\F, T_1)$ immersed in $(\G, T_2)$.
\end{defi}
We can then define our main classes of filtrations:
\begin{defi} \label{defi:standard_filtration}
Let $(\F, T)$ be a dynamical filtration on $\bfX := (X, \A, \mu, T)$. It is of product type if there is a sequence $(\C_n)_{n \leq 0}$ of mutually independent factor $\s$-algebras such that 
$$\forall n \leq 0, \; \F_n = \bigvee_{k \leq n} \C_k \; \text{ mod } \, \mu.$$
It is standard if it is immersible in a product type dynamical filtration.
\end{defi}

We chose our definitions to get the following properties:
\begin{prop} \label{prop:standard_filt_product_type}
We have
\begin{enumerate}
\item If $(\F, T)$ is of product type, then every extension $\F_{n+1} \arr \F_n$ is of product type.
\item If $(\F, T)$ is standard, then every extension $\F_{n+1} \arr \F_n$ is standard.
\end{enumerate}
\end{prop}
\begin{proof}
It follows from the definitions.
\end{proof}
Below, we use this proposition to build a non-standard filtration using a non-standard extension.

In the static case (i.e. when $T = \mathrm{Id}$), the existence of super-innovations implies that the standardness of a filtration is an asymptotic property \cite[Proposition 3.38]{Laurent_standardness-cosiness}: $(\F_n)_{n \leq 0}$ is standard if and only if there is a $n_0 \leq 0$ such that $(\F_n)_{n \leq n_0}$ is standard. In the dynamical case, the existence of extensions without super-innovations puts that asymptotic property in question. Then, the existence of non-standard extensions shows that standardness is not an asymptotic property for dynamical filtrations (using Proposition \ref{prop:standard_filt_product_type}). This is the main guideline for what we do next.

One of the main goals in the study of dynamical filtration is to find a standardness criterion, and for that purpose, dynamical I-cosiness was introduced in \cite{these_PL}. It relies on the notion of \emph{real time joinings} of filtrations: by that we mean a system $\bfZ := (Z, \C, \l, R)$ and a pair $((\F', R) , (\F'', R))$ defined on $\bfZ$ such that both $(\F', R)$ and $(\F'', R)$ are isomorphic to $(\F,T)$ and immersed in $(\F' \vee \F'', R)$.

A dynamical filtration $\F$ is I-cosy if for every $\F_0$-measurable random variable $\xi$ taking values in a compact metric space $(E, d)$ and every $\delta > 0$ there exists an integer $n_0 \leq 0$ and a real time joining $((\F', R), (\F'', R))$ such that $\F'_{n_0} \indep \F''_{n_0}$ and 
$$\bbE[d(\xi', \xi'')] \leq  \delta,$$
where $\xi'$ and $\xi''$ are the respective copies of $\xi$ in $\F_0'$ and $\F_0''$. In the static case, it is known that I-cosiness is equivalent to standardness (see \cite[Theorem 4.9]{Laurent_standardness-cosiness}). In the dynamical case that is of interest to us here, it was proved in \cite{these_PL} that standard dynamical filtrations are I-cosy, but the converse was left as an open question. The purpose of this section is to prove, using a non-standard extension, that the converse is not true: in the dynamical setting, I-cosiness is necessary but not sufficient for a dynamical filtration to be standard.

\begin{prop} \label{prop:non-stadard_I-cosy}
There exists a non-standard and $I$-cosy dynamical filtration.
\end{prop}
\begin{proof}
Let $\pi: \bfX \to \bfY$ be a factor map yielding a non-standard extension (we know it exists from Theorem \ref{thm:hyper-confined_extension}). Take a sequence $(\bfV_n)_{n \leq -2}$ of non-trivial dynamical systems and set 
$$\bfZ = (Z, \C, \rho, R) := \left(\bigotimes_{n \leq -2} \bfV_n\right) \otimes \bfX.$$
We consider the filtration defined by
$$\F_n := \left\{
\begin{array}{ll}
\bigvee_{k \leq n} \bfV_k & \text{if } n \leq -2\\
\bigvee_{k \leq -2} \bfV_k \vee \s(\pi) & \text{if } n = -1\\
\bigvee_{k \leq -2} \bfV_k \vee \bfX & \text{if } n = 0\\
\end{array} \right.$$
Since, $\bfX \overset{\pi}{\to} \bfY$ is not standard, it is easy to check that the extension
$$\bigotimes_{n \leq -2} \bfV_n \otimes \bfX \overset{\tilde\pi}{\to} \bigotimes_{n \leq -2} \bfV_n \otimes \bfY$$
 is not either. Therefore, Proposition \ref{prop:standard_filt_product_type} yields that $(\F, R)$ is not standard. 

We will use and slightly adapt the argument used in \cite{article_PL} to show that product type filtrations are I-cosy to show that $\F$ is also I-cosy. Let $\xi$ be a $\F_0$-measurable random variable taking values in a compact metric space $(E, d)$ and $\delta > 0$. There exist $n_0 \leq -2$ and a $\left(\bigvee_{n_0+1 \leq k \leq -2} \bfV_k \vee \bfX\right)$-measurable $\tilde \xi$ such that 
$$\bbE[d(\xi, \tilde\xi)] \leq \delta/2.$$
We now introduce our joining: we set the system
$$\bfW = (W, \D, \gamma, Q) := \bigotimes_{n \leq n_0} \bfV'_n \otimes \bigotimes_{n \leq n_0} \bfV''_n \otimes \bigotimes_{n_0+1 \leq n \leq -2} \bfV_n \otimes \bfX,$$
and denote $(\F', \F'')$ the copies of $\F$ on $\bfW$. By that, we mean that
$$\F'_n := \left\{
\begin{array}{ll}
\bigvee_{k \leq n} \bfV'_k & \text{if } n \leq n_0\\
\F'_{n_0} \vee \bigvee_{n_0 < k \leq n} \bfV_k & \text{if } n_0 < n \leq -2\\
\F'_{-2} \vee \s(\pi) & \text{if } n = -1\\
\F'_{-2} \vee \bfX & \text{if } n = 0\\
\end{array} \right.,$$
and a similar definition for $\F''$.
Clearly, $\F'_{n_0}$ and $\F''_{n_0}$ are independent and $\tilde\xi' =  \tilde\xi''$, which yields 
$$\bbE[d(\xi', \xi'')] \leq \bbE[d(\xi', \tilde\xi')] + \bbE[d(\xi'', \tilde\xi'')] \leq \delta.$$

We now only need to check that $((\F', Q), (\F'', Q))$ is a real time joining, i.e. for every $n \leq -1$:
$$\F'_{n+1} \indep_{\F'_n} \F''_n \; \text{ and } \; \F''_{n+1} \indep_{\F''_n} \F'_n.$$
For $n \leq -2$ we get the relative independent condition like in the product type case. We now check that $\F'_{0} \indep_{\F'_{-1}} \F''_{-1}$, which reduces to $\bfX' \indep_{\bfY'} \bfY''$. However, this is clearly true because $\bfY' = \bfY''$. 
\end{proof}

Here we exploit the strong structure of some specific extension to get a non-standard filtration. Therefore it is natural to ask:
\begin{quest}
Is there an I-cosy dynamical filtration such that each extension $\F_{n+1} \arr \F_n$ is of product type, but which is still not standard ?
\end{quest}

\begin{description} [leftmargin=*] \item[Acknowledgments.]
The author thanks Thierry de la Rue and Emmanuel Roy for their supervision, diligent reading and suggestions. We are also grateful to Christophe Leuridan and Mariusz Lema\'nczyk for taking interest in this work and giving insightful remarks.
\end{description}

\bibliographystyle{plain}
\bibliography{biblio_confined_extensions_paper}

\end{document}